\documentclass[a4paper,fleqn]{cas-sc}

\usepackage[authoryear,longnamesfirst]{natbib}

\usepackage[utf8]{inputenc}
\usepackage{graphicx}
\usepackage{amsmath}
\usepackage{amssymb}
\usepackage{amsthm}
\usepackage{dirtytalk}

\usepackage{quiver}
\usepackage{todonotes}

\newtheorem{thm}{Theorem}
\newtheorem{prop}[thm]{Proposition}%
\newtheorem{lem}[thm]{Lemma}
\newtheorem{cor}[thm]{Corollary}

\newtheorem{rem}{Remark}%

\newtheorem{dfnt}{Definition}%

\def\N{\mathbb{N}}

\newcommand{\cI}{\mathcal{I}}

\newcommand{\bA}{\mathbb{A}}
\newcommand{\bB}{\mathbb{B}}
\newcommand{\md}{\mathrm{UMod}}
\newcommand{\qmod}{\mathrm{Mod}}
\newcommand{\fmod}{\mathrm{FMod}}
\newcommand{\smod}{\overline{\mathrm{Mod}}}
\newcommand{\amod}{\mathrm{AlgMods}}
\newcommand{\merg}{\mathrm{Merge}}
\newcommand{\G}{\Vec{H}}
\newcommand{\K}{\mathcal{K}}
\newcommand{\cS}{{\mathcal{S}}}
\newcommand{\cT}{{\mathcal{T}}}
\def\Z{\mathbb{Z}}

\begin{document}

\let\WriteBookmarks\relax
\def\floatpagepagefraction{1}
\def\textpagefraction{.001}
\shorttitle{Homological algebra over non-unital algebras}
\shortauthors{E. Medioni et~al.}

\title[mode = title]{Homological algebra over non-unital rings and algebras, with applications to $(\infty, 1)$-categories}


\author[1]{Eric Goubault}[orcid=0000-0002-3198-1863]
\ead{eric.goubault@polytechnique.edu}
\credit{Conceptualization, Writing - Review \& Editing, Supervision.}

\author[1]{Eliot M\'edioni\corref{cor1}}
\ead{eliot.medioni@student-cs.fr}

\cortext[cor1]{Corresponding author}
\credit{Conceptualization, Writing - original draft.}

\affiliation[1]{organization={LIX, CNRS and Institut Polytechnique de Paris},
addressline={},
postcode={91128},
city={Palaiseau},
country={France}}

\begin{abstract}
The article is developing homological algebra in modules over non-unital rings and algebras. The main application is the definition and study of (directed) homology of $(\infty,1)$-categories and of directed spaces, including relative homology and its exact sequence. 
\end{abstract}

\begin{keywords}


\MSC[2020] 55N35 18E10 13C60
\end{keywords}


\maketitle



\section{Introduction}

The category of modules over rings is the archetype of the category in which to carry out homological algebra. Indeed this category is abelian, and all abelian categories can be represented as a concrete category of modules over some ring, by the Freyd-Mitchell embedding theorem \cite{freyd1964abelian}. This goes for modules over {\em unital rings}, and a natural question is to know whether this holds as well for modules over non-unital rings, for which there exists several versions, e.g. firm, closed and $s$-unital modules, \cite{quillen,thfmn,tominaga}. 
Under tame assumptions, all those notions because equivalent and give rise to an abelian category, making them prime candidates for homological theories. 


This question is not uniquely of an abstract nature. In numerous contexts, one would actually need a homology theory over non-unital rings, or even non-unital algebras. 

For instance, in topological data analysis, persistence modules \cite{Carlsson,multidimpersistence,Bubenik} are used to define persistence homology, and are, by Morita equivalence, the same as persistence modules over their path algebras \cite{assocalg}. In case the underlying filtration or  multi-dimensional filtration has infinitely many strata, or in case the underlying category has an infinite number of objects, the corresponding path algebra has an infinite number of idempotents. In this case, the path algebra (and in particular its underlying ring) has no natural unit, and persistence modules are to be considered over a non-unital algebra (and underlying ring). 

In directed topology \cite{thebook,grandisbook}, topological spaces under study are ``directed'', meaning they have ``preferred directions'', for instance a cone in the tangent space, if we are considering a manifold, or the canonical local coordinate system in a precubical set. A directed homology theory that accounts for directed features has been introduced in \cite{eric}, as particular persistence modules over the path algebra of the underlying quiver of a precubical set, building on natural homology theories \cite{Dubut}. Once again, if the precubical set has an infinite number of points, this path algebra has to be considered non-unital. This explains the limitation of the approach of \cite{eric} which only considers finite precubical sets, and the impossibility to define a directed homology theory for more general directed spaces, let alone $(\infty,1)$-categories. 

In this paper, we begin by recapping the somewhat scattered knowledge about modules over non-unital rings and adapt it to (bi-)modules over non-unital algebras. We then proceed to apply it to a (directed) homology theory for simplicially enriched categories, which constitute a model for $(\infty,1)$-categories in the case where the hom-simplicial-sets are Kan. 
We prove that, even in this non-unital context, relative homology can be defined and corresponding exact sequences, proven. We make the conditions for these exact sequences to hold more explicit in the case of directed spaces, extending the approach of \cite{eric}, at the end of the article. 

\paragraph{Organization of the paper}

We start in Section \ref{sec:background} by reviewing the scattered litterature about modules over non-unital rings and algebras. An important ingredient in understanding the properties of categories of modules over non-unital algebras is the unitalization functor, for algebras and modules, that we study in great detail in Section \ref{sec:unitalization}. 
The properties of firm, closed and $s$-unital modules (over non-unital rings and algebras) are then recapped for some, and studied (as new material) for others, in Section \ref{sec:categories}. Among these, the extension of scalars functor, Section \ref{sec:extensionscalar}, is fully studied in this non-unital case, as this will prove instrumental in the definition and properties of relative homology and the relative homology exact sequence. 

In most applications we are envisioning, such as the ones in directed topology, there is a need for 
bimodules instead of (left or right) modules. In the classical unital case, there is no particular difficulty in going from (left or right) modules to bimodules since $A$-$B$-bimodules can be seen as $A{\otimes}_\Z B^{op}$ (left) modules. 
This is not always the case for bimodules over non-unital rings or algebras, so we show directly that $s$-unital bimodules over fixed $s$-unital algebras form an abelian category. 

We then proceed to the main application of this article, the definition and study of a directed homology theory of $(\infty,1)$-categories, Section \ref{sec:infinityone}, including relative homology and the relative homology exact sequence, largely generalizing the results of \cite{eric}. 

Finally, we apply the directed homology theory of $(\infty,1)$-categories to directed spaces, that we briefly review, and for which we can give more explicit constructions of relative pairs. 

\paragraph{Main results}

The article intends to act both as a survey of known, folklore theorems about modules over non-unital rings and algebras, that are very much scattered in the literature, and as a proposal for new homology theories, or generalization of existing (directed) homology theories such as the one developed in \cite{eric}. Most notably,
\begin{itemize}
\item We are reviewing the notions of modules over non-unital rings and algebras, Section \ref{sec:background}. In particular, we are trying to make explicit all conditions and constructions that we need, that are generally otherwise not fully spelled out. 
\item We prove that there is an isomorphism of categories between the category of unital (bi)modules over the unitalization of algebra $\bA$ (and $\bB$) and the category of (non-unital) (bi)modules over non-unital algebra $\bA$ (and $\bB$), Theorem \ref{thm:3},
\item We prove that restriction of scalars, as well as extension of scalars, can be defined even in the non-unital context, 
\item We prove that $s$-unital modules over $s$-unital algebras, a convenient case of firm modules over non-unital algebras, form an abelian category, Corollary \ref{cor:10},
\item We define a convenient category of $s$-unital $\bA$-$\bB$-bimodules and prove that it is an abelian category, and that it is also the category of modules over the (unital) algebra $\hat{\bA}\otimes \hat{\bB}^{op}$ ($\hat{\bA}$, resp. $\hat{\bB}$, being the unitalization of non-unital algebras $A$ and $B$), Proposition \ref{prop:11}, 
\item Given an $(\infty,1)$-category, seen as a Kan simplicial enriched category $\cS$, we define a path algebra $R[\cS]$ which we show is $s$-unital, Proposition \ref{prop:12}, 
\item Chain complexes that are used to define a directed homology theory for $\cS$ are shown to define $s$-unital bimodules, Theorem \ref{thm:13},
\item We prove that under some conditions, the definition of directed homology of $(\infty,1)$-categories is functorial, Theorem \ref{thm:15}, Corollary \ref{cor:17} and that directed homology is an invariant for Kan simplicially enriched categories, Corollary \ref{cor:18},
\item We prove that there is an exact sequence for relative homology of $(\infty,1)$-categories, Section \ref{sec:exact},
\item We particularize these constructions for directed spaces seen as $(\infty,1)$-categories, Theorem \ref{thm:21}, and give simple conditions for the relative homology exact sequence to hold, Proposition \ref{prop:24}.
\end{itemize}

\section{Background on modules and (nonunital) algebras}

\label{sec:background}

We first 
review here the literature about modules over nonunital algebras. The content of this section consists mostly in already known results and folklore theorems, that we recap and collect in order to be self contained. 

\subsection{Nonunital algebras}

\begin{dfnt}[Rings and rings without unit]
Let $(A,+)$ be an abelian group. If $\times$ is a binary operation on $A$ that is associative and distributive over $+$, we say that $(A,+,\times)$ is a \textbf{rng} (or ``rung") 
or ring without unit. 

If there exists $a \in A$ s.t. $\forall x \in A, \; a\times x = x \times a = x$ then $a$ is unique, denoted by the symbol $1$, and $(A,+ , \times)$ is called a \textbf{ring} or unital ring.  

A \textbf{rng homomorphism} from an rng $(A,+,\times)$ to an rng $(B,+, \times)$ is a function $f: A \to B$ that is compatible with $+$ and $\times$. A ring homomorphism between two rings is a rng homomorphism that preserves the unit element $1$.
\end{dfnt}

\begin{dfnt}[Algebras and unital algebras]
    Let $R$ be a commutative ring.
    An \textbf{algebra} on $R$ is a rng $(\bA ,+, \times)$ equipped with a function $R \times \bA \to \bA$, denoted by $.$, so that:
    \begin{enumerate}
        \item[(i)] $(\bA,+,.)$ is a left $R$-module
        \item[(iv)] The product $\times$ is $R$-bilinear
    \end{enumerate}
    If the underlying rng is unital, then the algebra is said to be \textbf{unital}.
    
    An algebra homomorphism $\bA \to \bB$ is a function $\phi : \bA \to \bB$ such that
    \begin{enumerate}
        \item[(i)] $\phi$ is a rng homomorphism
        \item[(ii)] $\phi$ is a $R$-module homomorphism
    \end{enumerate}
\end{dfnt}

In what follows, we will never assume an algebra to be unital unless otherwise stated.

\begin{rem}
This is not a conventional definition, and they are many different notions of nonunital algebras in the literature. Nonunital algebras $\bA$ are often defined as a ring together with a nonunital ring morphism from $R$ onto the center of $\bA$ that satisfies some conditions. However, in our particular case, we need $\bA$ to always inherit a $R$-module structure, which would not be the case for these nonunital algebras. Our definition can also be motivated by the following proposition.
\end{rem}

\begin{rem}
Since $R$ is commutative, choosing between a left or right module structure is strictly a matter of notations. When we have one of either kind, we can always interpret it as a bimodule structure by defining $r.a$ as $a.r$ (or the opposite) for any $r\in R$ and $a \in \bA$.
\end{rem}

\begin{rem}
$R$ itself is a unital $R$-algebra.
\end{rem}

\begin{prop}
\begin{enumerate}
    \item Let $(A, +, \times, .)$ be a unital $R$-algebra and $\bA$ a sub-$R$-module of $(A, +, .)$ that is stable by $\times$.
    Then $\bA$ together with the induced laws is a an $R$-algebra.
    \item Any (nonunital) $R$-algebra $(\bA, +, \times, .)$ arises in this manner as a sub-$R$-module of a unital algebra $A$ such that $\times$ is induced from the product in $A$.
\end{enumerate}
\end{prop}


The proof of the first point is straightforward and the second results almost immediately from the following construction:

\begin{rem}
\label{rem:not}
In what follows, we will generally reserve the letters $A, B, \ldots$ for unital algebras, and $\bA, \bB,\ldots$ for non-unital algebras.
\end{rem}

\begin{dfnt}[Unitalization]
\label{def:algunitalization}
Let $R$ be a commutative ring (with unit).
Let $\bA$ be an $R$-algebra.

Let us endow the abelian group $\widehat{\bA} = (\bA,+) \times (R,+)$ with the product given by:

$$(a,r)\times (b,s) = (ab + r.b + a.s, rs)$$

Then $\widehat{\bA}$ is a unital algebra on $R$ called the \textbf{unitalization} of $\bA$.
Its unit is $(0,1)$.
\end{dfnt}

This construction is also sometimes called called the \textit{Dorroh extension} in reference to a 1932 article studying the case where $R$ is the ring of integers \cite{Dorr}.

\subsection{Modules over (nonunital) algebras}

We will now define and study the properties of modules over nonunital algebras. 

\begin{dfnt}
\label{def:modules}
    Let $R$ be a ring and $\bA$ some $R$-algebra.
    
    Let  $(M,+)$ be an abelian group,
 together with a function $._{_\bA} : \bA \times M \to M$ (resp. $M \times \bA \to M$) and a function $._{_R} : R \times M \to M$ (resp. $M \times R \to M$) such that:
 \begin{enumerate}
    \item $\forall m \in M, \forall a,b \in \bA, \quad
    (a+b)._{_\bA}m = a._{_\bA}m + b._{_\bA}x$ (resp. $m.(a+b) = m.a + m.b$)
    \item $\forall m, n\in M, \forall a \in \bA, \quad
    a._{_\bA}(m+n) = a._{_\bA}m + a._{_\bA}n$ (resp. $(m+n).a = m.a + n.a$)
    \item $\forall a,b \in \bA, \forall m \in M, \quad
    a._{_\bA}(b._{_\bA}m) = (ab)._{_\bA}m$ (resp. $(m.a).b = m.(ab)$)
    \item $(M,+,._{_R})$ if a left (resp. right) module (in the usual, unital sense) over the unital ring $R$
    \item $\forall r \in R, \forall a \in \bA, \forall m \in M$,
    \[r._{_R}(a._{_\bA} m) = (r.a)._{_\bA} m = a._{_\bA} (r._{_R} m)\]
    (resp. \((m._{_\bA} a)._{_R} r = m._{_\bA} (r.a) = (m._{_R} r)._{_\bA} a\))
    
 \end{enumerate}

Then we call $(M,+,.)$ a \textbf{left-module} (resp. right-module) over $\bA$.
\end{dfnt}

The reasoning behind this definition will become clearer after having studied the structural properties of the unitalization, and particularity its effect on module categories.

\begin{rem}
The relation $(r.a)._{_\bA} m = a._{_\bA} (r._{_R} m)$ is best understood by defining a right-$R$-module structure on $M$ by $m.r := r._{_R} m$, 
so that the relation can be rewritten as:
\[a.(r.m) = (a.r).m\]
\end{rem}

When $\bA$ is unital, then 
\[
r._{_R} m = r._{_R} (1_{_\bA}._{_\bA} m) = (r.1_{_\bA} )._{_\bA} m
\]
\noindent and thus $._{_R}$ is uniquely determined by $._{_\bA}$. This is the property that makes the different definitions of a unital $R$-algebra coincide.

\begin{dfnt}
If $(M, +, ._{_R}, ._{_\bA})$ is a left module over an $R$-algebra $\bA$ and $(M, +, ._{_S}, ._{_\bB})$ is a right module over an $S$-algebra $\bB$ such that: 
\begin{itemize}
    \item \( \forall a \in \bA, \forall b \in \bB, \
    a._{_\bA}(m._{_\bB} b) = (a._{_\bA} m)._{_\bB} b\)
    \item \(\forall a \in \bA, \forall s \in S, \
    a._{_\bA} (m._{_S} s) = (a._{_\bA} m)._{_S} s\)
    \item \(\forall r \in R, \forall b \in \bB, \
    r._{_R} (m._{_\bB} b) = (r._{_R} m)._{_\bB} b\)
    \item \(\forall r \in R, \forall s \in S, \
    r._{_R} (m._{_S} s) = (r._{_R} m)._{_S} s\)
\end{itemize}
Then $(M, +, ._{_\bA}, ._{_\bB})$ is called a $\bA$-$\bB$-bimodule.
\end{dfnt}

In the following, for simplicity's sake, we will generally drop the subscript of the action of the algebra $._{_\bA}$ and $._{_\bB}$ when the algebras are clear in the given context.

\begin{dfnt}
 A left (resp. right) \textbf{module homomorphism} from $M$ to $N$ two left (resp. right) modules on $\bA$ is a $R$-linear function $f: M \to N$ such that:
 \[f(\lambda .x + \mu .y) = \lambda .f(x) + \mu .f(y)\]
 (resp. $f(x.\lambda + y.\mu) = f(x).\lambda + f(y).\mu$)for all $\lambda, \mu \in \bA$ and $x, y \in M$.
 
 A $\bA$-$\bB$-bimodule homomorphism from $M$ to $N$ is a left $\bA$-module homomorphism $f : M \to N$ that is also a right $\bB$-module homomorphism.
\end{dfnt}


The category of $\bA$-$\bB$-bimodules will be denoted by $_\bA \qmod_\bB$. 

When $A$ and $B$ are unital algebras, the category of \textit{unital} $A$-$B$-bimodules, that is $A$-$B$-bimodules in the usual sense over the unital rings $A$ and $B$, will be denoted as $_A \md_B$.

It is clear that an algebra on a commutative ring $R$ has a canonical $R$-$R$-bimodule structure with added commutativity properties: for any $a$ in the algebra and any $r \in R$, the exterior product $r . a$ (resp. $a.r$) is defined as $r \times a$ (resp. $a \times r$) and we have the property that $a.r = r.a$.

We will also need in the sequel the category of bimodules over \textit{varying} algebras: 

\begin{dfnt}[Category of bimodules over algebras]
If $R$ and $S$ are commutative (unital) rings, define the following category:
\begin{itemize}
    \item The objects are the triples $(\bA, M, \bB)$ where $\bA$ is an $R$-algebra, $\bB$ is an $S$-algebra and $M$ is a $\bA-\bB$-bimodule
    \item The morphisms from $(\bA, M, \bB)$ to $(\bA', M', \bB')$ are the triples $(f,g,h)$ where $f : \bA \to \bA'$ is an $R$-algebra morphism,
    $h : \bB \to \bB'$ is an $S$-algebra morphism and $g : (M, +) \to (M', +)$ is a group homomorphism such that:
    \[\forall a \in \bA, \forall m \in M, \forall b \in \bB,
    h(a.m.b) = f(a).h(m).g(b)\]
    \item Composition is the usual composition of functions
\end{itemize}
This category is the \textbf{category of bimodules over $\mathbf{R}$-algebras on the left and $\mathbf{S}$-algebras on the right}
, that we will denote by $_R \amod _S$. 
\end{dfnt}

\section{Unitalization of algebras and modules, and their properties}

\label{sec:unitalization}

This paragraph aims at showing that unitalization is not an arbitrary operation, in that it is natural in two very strong meanings of the term.

\begin{thm}
    Let $U$ be the functor from $\mathrm{Alg}_R$ to $\mathrm{UAlg}_R$ such that:
    \begin{itemize}
        \item $U(\bA) = \widehat{\bA}$ for $\bA$ an object of $\mathrm{Alg}_R$ (see Definition \ref{def:algunitalization})
        \item $U(f) : (a,r) \mapsto (f(a), r)$ for $f : \bA \to \bB$ in $\mathrm{Alg}_R$
    \end{itemize}
    
    Then $U$ is left adjoint to the forgetful functor from $\mathrm{UAlg}_R$ to $\mathrm{Alg}_R$.
\end{thm}

This means that the unitalization is functorial and is the most general way to add a unit to a given algebra. This of course extends to right algebras.

\begin{proof}
It is clear that $U$ is functorial. Denote by $I$ the forgetful functor.

Let us show that $(U,I)$ is an adjoint pair by giving explicitely the natural transformations satisfying the unit-counit equations.

In this proof, as already noted in Remark \ref{rem:not} we will use different fonts for naming unital and nonunital algebras, for the sake of clarity: for instance, in the sequel, $\bA$ will denote a nonunital algebra, whereas $A$ will denote a unital algebra. 

To any (nonunital) $R$-algebra $\bA$ associate the $R$-algebra homomorphism 
\[\eta_\bA : \bA \to IU(\bA) \]
defined by:
\[\eta_{\bA} (x) =
(x, 0)\]
To any unital $R$-algebra $A$ associate the unital algebra homomorphism 
\[\epsilon_{A} : UI(A) \to A\]
defined by
\[
\epsilon_{A} : (x,r) \mapsto  x + r.1_A
\]

Then $\eta$ is a natural transformation from $\mathrm{Id}_{\mathrm{Alg}_R}$ to $UI$ and $\epsilon$ is a natural transformation from $IU$ to $\mathrm{Id}_{\mathrm{UAlg}_R}$.

Moreover, for any (nonunital) algebra $\bA$,
\[
(\epsilon U \cdot U \eta)_\bA
= \epsilon_{U(IU(\bA))} \circ U(\eta_\bA) \]
and for any $x \in \bA$ and $r \in R$, $U(\eta_\bA) (x,r) = (\eta_\bA (x), r) = ((x,0), r)$ so 
\[ (\epsilon U \cdot U \eta)_\bA (x,r)
= \epsilon_{UIU(\bA)}( ((x,r), 0) )
= (x,r)
\]
For any unital algebra $A$:
\[
(I \epsilon \cdot \eta I)_{A} =
I(\epsilon_{UI(A)}) \circ \eta_{I(A)}
\]
so for any $x \in A$:
\[
(I \epsilon \cdot \eta I)_{A} (x)
= I(\epsilon_{UI(A)}) (x,0)
= \epsilon_{UI(A)} ((x,0))
= x
\]
As such, $(U,I)$ is an adjoint pair with unit $\eta$ and counit $\epsilon$

\end{proof}

Using the previous construction, we show that the category of nonunital modules over nonunital algebras is isomorphic to the category of unital modules over unital algebras:

\begin{thm}
\label{thm:3}
     Let $R$ and $S$ be commutative (unital) rings, $\bA$ an algebra on $R$ and $\bB$ an algebra on $S$. 

     To every $M \in {}_\bA\qmod_\bB$, associate $F(M) \in {} _{\widehat{\bA}} \md _{\widehat{\bB}}$ to be $M$, with the  addition law, but whose external laws are given by:
    \[
    (a,r).m = a._{_\bA} m + r._{_R} m \quad
    \text{ and } \quad
    m.(b, s) = m._{_\bB} b + m._{_S} s
    \]
    
     To every $f : M \to N$ in $_\bA\qmod_\bA$, associate $Ff \in {} _{\widehat{\bA}} \md _{\widehat{\bB}}$ which is equal to $f$. 
     Then $F$ is an isomorphism of categories.
\end{thm}

This property is the main justification for the definition of module over an algebra used previously. In fact, in a 1996 article, Quillen defines modules over a nonunital algebra $\bA$ as the modules over some arbitrarily chosen unital algebra $\bB \supset \bA$ that verify some additional conditions \cite{quillen}. He then shows that the choice of $\bB$ does not matter up to category equivalence. However, we prefer to express our conditions using only simple algebraic relations to make computational aspects clearer.

Note that Theorem \ref{thm:3} explains why $r.(a.m)$ should be equal to $a.(r.m)$ in the last axiom of our definition of a module over an algebra: in a unital algebra over $R$, the subring $R.\{1\}$ is contained in the center of the algebra, so any $(0,r)$ for $r \in R$ must commute with any $(a, 0)$ for $a \in \bA$, and as such $a.(r.m) = (a,0).((0,r).m) = (0,r).((a,0).m) = r.(a.m)$.

\begin{proof}
First of all, let us show that $F$ is well defined. Let us take $M \in _{\widehat{\bA}} \qmod _{\widehat{\bB}}$. We first need to show that $F(M)$ is a bimodule. 

It is clear that for all $x \in \widehat{\bA}$ and $y \in \widehat{\bB}$, $m \mapsto x.m.y$ is linear. 
Moreover, $1_{\widehat{\bA}}.m = (0,1_R).m = m$ and $m.1_{\widehat{\bB}} = m.(0, 1_S) = m$. 
Also, 
\begin{align*}
(a_1,r_1).((a_2,r_2).m )  &= (a_1,r_1).( a_2.m + r_2.m)
\\&= a_1.( a_2.m + r_2.m) +r_1. ( a_2.m + r_2.m)
\\&= a_1.(a_2.m) + a_1.(r_2.m) + r_1.(a_2.m) + r_1.(r_2.m)
\\&= (a_1 a_2) .m + (a_1.r_2).m + (r_1.a_2) .m + (r_1 r_2).m
\\&= (a_1 a_2 + a_1.r_2 + r_1.a_2 + r_1 r_2).m
\\&= ((a_1,r_1)\times (a_2,r_2)).m 
\end{align*}
\noindent and similarly
\begin{align*}
(m.(b_1, s_1)).(b_2, s_2)
&=(m.b_1 + m.s_1).(b_2, s_2)
\\&= m.(b_1 b_2 + s_1 .b_2 + b_1 .s_2 + s_1 s_2)
\\&= m.( (b_1, s_1) \times (b_2, s_2))
\end{align*}

Finally, 
\begin{align*}
((a, r).m).(b,s)
&=(a.m + r.m).(b, s)
\\&= (a.m).(b,s) + (r.m).(b,s)
\\&= (a.m).b + (a.m).s + (r.m).b + (r.m).s
\\&= a.(m.b) + a.(m.s) + r.(m.b) + r.(m.s)
\\&= a.(m.b + m.s) + r.(m.b + m.s)
\\&= a.(m.(b,s)) + r.(m.(b,s))
\\&= (a,r).(m.(b,s))
\end{align*}
\noindent so $F(M)$ is a well-defined bimodule.
Moreover, the functoriality of $F$ is trivial.

To conclude, let $G$ be the functor such that:
\begin{itemize}
    \item For every object $M$ of $_{\widehat{\bA}} \qmod _{\widehat{\bB}}$,   $\; G(M)$ is the $\bA$-$\bB$-bimodule with $(M,+)$ as its underlying additive group and with the product law given by $$a.m.b = (a,0).m.(b,0)$$ 
        \item For every $g : M \to N$ in $_{\widehat{\bA}} \qmod _{\widehat{\bB}}$,   $\; Gg = g$
\end{itemize}

Then $G$ is clearly functorial and reciprocal of $F$.
\end{proof}

\begin{cor}
$\;_\bA\qmod_\bB$ is an abelian category enriched over $\Hat{\bA} \otimes \Hat{\bB}^{\text{op}}$-modules.
\end{cor}

\begin{proof}
If $M$ and $N$ are objects in $_\bA\qmod_\bB $, one can clearly define a left $\Hat{\bA} \otimes \Hat{\bB}^{\text{op}} $-module structure on $_\bA\qmod_\bB (M,N)$ by:
\begin{align*}
\forall f \in \,_\bA\qmod_\bB (M,N), \;
\forall (a,r) \in \Hat{\bA}, \;
\forall (b,s) \in \Hat{\bB},& \\
\left((a,r) \otimes (b,s) \right).f : \; \; m \mapsto 
a.f(m).b + a.f(m).s &+ r.f(m).b + r.f(m).s
\end{align*}
And this makes composition of applications bilinear. 

It is then clear that $F$ is a $\Hat{\bA} \otimes \Hat{\bB}^{\text{op}} $-linear functor, and since it is an isomorphism we obtain the desired result.
\end{proof}

\begin{rem}
Since $F$ is a linear functor of $\Hat{\bA} \otimes \Hat{\bB}^{\text{op}}$-abelian categories on top of being an isomorphism, it is an exact functor. As such, applying $F$ to the homology sequence of a module chain complex is the same as taking the homology sequence of the image by $F$ of the module chain complex. 
\end{rem}

\section{On categories of closed, firm and s-unital modules}

\label{sec:categories}

Take some unital commutative ring $R$. We can show that the category of modules over the $R$-algebra $R$ as defined earlier (in Definition \ref{def:modules}) is not equivalent to the category of modules over the ring $R$ in the usual sense (see introduction of \cite{quillen}).  

This is one of the reasons why one might question whether $_\bA\qmod$ (or more generally $_\bA\qmod _\bB$) is a good category to do homological algebra in.

An article of Quillen \cite{quillen} answers this concern by defining three reasonable categories of modules over a nonunital ring or algebra, that are more natural in some ways and coincide under mild hypotheses. We will go over his approach after having defined the tensor product of modules over algebras.

\subsection{Nonunital tensor product}
Quillen defines the tensor product of modules over a rng using an arbitrarily chosen embedding of the rng into a ring. If that embedding is taken to be the canonical embedding given by unitalization, then this gives rise to the following construction. 

Taking some $R$-algebra $\bA$, a right $\bA$-module $M$ and a left $\bA$-module $N$, we start by seeing $M$ and $N$ as modules $\Tilde{M}$ and $\Tilde{N}$ over the unitalization $U(\bA)$ using the category isomorphism of the previous section. Then, we consider the tensor product
\[
\Tilde{M} \otimes_{U(\bA)} \Tilde{N}   
\]
This amounts to taking the free abelian group generated by $M \times N$, and quotient it out by the subgroup generated by elements of the following form, for $m,m' \in M$, $n,n' \in N$ and $(a,r) \in U(\bA) = \bA \oplus R$:
\begin{itemize}
    \item $(m+m', n) - (m,n) - (m',n)$
    \item $(m, n+n') - (m,n) - (m,n')$
    \item $(m.(a,r), n) - (m, (a, r).n)$
\end{itemize}
Since $(m.(a,r), n) - (m, (a, r).n) = (m.a, n) - (m, n.a) + (m.r, n) - (m, r.n)$, we can equivalently replace elements of the form in the last point above, into elements of the following form:
\begin{itemize}
    \item $(m.a, n) - (m, a.n)$ for $a \in \bA$
    \item $(m.r, n) - (m, r.n)$ for $r \in R$
\end{itemize}

The canonical projection $M \times N \to M \otimes N$ is, as usual, denoted by $\otimes$.

This gives us an expression of the universal property of the tensor product that does not reference the unitalization in any way.

\begin{dfnt}
Let $G$ be a commutative group. A function $f : M \times N \to G$ so that, for any $m,m' \in M$, $n,n' \in N$ and $a\in \bA$ and $r \in R$:
\begin{enumerate}
    \item $f(m+m',n) = f(m,n) + f(m',n)$
    \item $f(m, n+n')=f(m,n) + f(m,n')$
    \item $f(m.a, n) = f(m,a.n)$
    \item $f(m.r, n) = f(m,r.n)$
\end{enumerate}
is said to be a $\bA-$\textbf{balanced bi-additive function}. 
\end{dfnt}

This tensor product is clearly characterized by the following universal property (that does not appear to be stated in previous literature):
\begin{prop}[Universal property of the tensor product]
For every commutative group $G$ and any $\bA$-balanced bi-additive function $f: M \times N \to G$,
there exists a unique map $\overline{f} : M \otimes N \to G$ such that 
\[f(m,n) = \overline{f} (m \otimes n)\]
for any $m \in M$ and $n \in N$.
\end{prop}

Since this tensor product is defined in terms of the usual one, all of the usual properties of the unital tensor product carry over, including the extra structure when tensoring bimodules,  associativity etc.

\subsection{Categories of firm and closed modules}
Let us consider an $R$-algebra $\bA$ (or a rng, since rngs are particular cases of $\mathbb{Z}$-algebras).\\

The first category we are considering 
is obtained by quotienting $_\bA \qmod$ as to identify with zero all ill-behaved objects. Consider the full subcategory of left $\bA$-modules $M$ verifying the following degeneracy property:
\[
\exists n \in \mathbb{N}, \quad
\bA^n M = 0
\]
This category is called the category of \textbf{nil modules}. It is a Serre subcategory \cite{serre}, 
which means that we can quotient it out to obtain an abelian category $\mathbf{_\bA \mathcal{M}}$ of non-degenerate modules. This is the most fundamental and well-behaved category of non-degenerate modules over non-unital algebras, but it is also not very explicit given because of the quotient operation.

The second category is the category of firm modules as defined below: 

\begin{dfnt}[Firm module]
The modules $M \in \;_\bA \mathrm{Mod}$ 
for which the natural map
\[
\mu : \bA \otimes M \to M, \quad a \otimes m \mapsto a.m
\]
is an isomorphism 
are said to be \textbf{firm}. 

We will denote by $_\bA\fmod$ the full subcategory of $_\bA \mathrm{Mod}$ of left firm $\bA$-modules.
\end{dfnt}

\begin{rem}
Note that for arbitrary $\bA$, the category of firm modules is not necessarily abelian \cite{gonzalez}.
\end{rem}

The last category is defined dually to firm modules:  

\begin{dfnt}[Closed module]
The modules $M \in \;_\bA \mathrm{Mod}$ such that the natural map
\[
\mu' : M \to \mathrm{Hom}(\bA,M), \quad
\mu'(m) = (a \mapsto a.m)
\]
is an isomorphism 
are said to be \textbf{closed}. The category of closed modules is the full subcategory of $_\bA \mathrm{Mod}$ whose objects are closed modules. 
\end{dfnt}

However, those three categories are often the same as shown by the following theorem established by \cite{quillen}. 

\begin{dfnt}
A rng $A$ is said to be \textbf{idempotent} if $A^2 = A$.
An algebra is idempotent when it is idempotent as a rng.
\end{dfnt}

\begin{thm}
If $\bA$ is an idempotent $R$-algebra 
then the category of firm modules, the category of closed modules, and the category of non-degenerate modules $_\bA \mathcal{M}$ are all equivalent abelian categories.
\end{thm}
All of this obviously extends to right modules by duality.

In this article, we will only use idempotent rings so we will systematically use firm modules and bimodules for convenience. 

\subsection{Extension of scalar for firm modules}
\label{sec:extensionscalar}

\begin{dfnt}
Let $g : \bA \to \bB$ be an $R$-algebra morphism.

Then for any $\bB$-module $M$, denote by $\cI(M)$ the $\bA$-module with the same underlying $R$-module but with external law given by
\[
a._\bA m = g(a)._\bB m
\]
Then for any $f : M \to N$, $I(f) = f$ is a $\bA$-module morphism $M \to N$ (since $f( g(a).m) = g(a).f(m)$), and we call the resulting functor $\cI_g$ the \textbf{restriction of scalars along $g$}.
\end{dfnt}

\begin{rem}
Taking $M$ to be firm, $\cI_g(M)$ is not necessarily a firm $\bA$-module. In fact, in the examples we will study in later sections, one can check that it is not the case because the canonical morphism is not even surjectve.
\end{rem}

\begin{prop}
Let $g : \bA \to \bB$ be an $R$-algebra morphism.

Equip $\bB$ with the right $\bA$-module structure defined by
\[
\begin{cases}
b.a = b \times g(a) \text{ for } a\in \bA, b \in \bB 
\\
b.r = r.b\text{ for } r \in R, b \in \bB 
\end{cases}
\]

\begin{itemize}
    \item For any $\bA$-module $M$, $\bB \otimes_\bA M$ has a natural (left) $\bB$-module structure given by
    \[
    \begin{cases}
    r.(b \otimes m) = (r.b) \otimes m \\
    b' .(b \otimes m) = (b'b) \otimes m
    \end{cases}
    \]
    \item If $\bB$ is a firm $\bB$-module, then for any $\bA$-module $M$ we have that $\bB \otimes_\bA M$ is a firm $\bB$-module.
\end{itemize}
\end{prop}

\begin{proof}
$\bullet$ The right $\bA$-module structure on $\bB$ is well-defined because $g$ is an $R$-algebra morphism. 

$\bullet$ It is clear that $\bB \otimes_\bA M$ is indeed a left $\bB$-module.

$\bullet$ Taking either 
$$
U = \bB \otimes_\bB ( \bB \otimes_\bA M)
\text{ and }
p_U : (b,b',m) \mapsto b \otimes (b' \otimes m)
$$ 
or
$$
U = (\bB \otimes_\bB \bB ) \otimes_\bA M
\text{ and }
p_U : (b,b',m) \mapsto (b \otimes b') \otimes m
$$ such that $p$ is the canonical $R$-module homomorphism $bB \times  \bB \times M \to N \to U$, 
$(U, p_U)$ verifies the following universal property:

    \textit{For any $R$-module $N$, every $R$-trilinear map $f : bB \times  \bB \times M \to N$ such that
    \begin{itemize}
        \item $f(b.b' , b'', m) = f(b, b'.b'', m)$ for any $b, b',b'' \in \bB, m \in M$
        \item $f(b , b'.a, m) = f(b, b', a.m)$ for any $a \in \bA, b,b' \in \bB, m \in M$
    \end{itemize}
    there is a unique $R$-module homomorphism $\Tilde{f} : U \to N$ such that 
    \[ 
    \Tilde{f} = f \circ p_U
    \]}

As such, there is a a canonical isomorphism 
\[
\phi : (\bB \otimes_\bB \bB) \otimes_\bA M 
\to \bB \otimes_\bB (\bB \otimes_\bA M) 
\]
such that 
$\phi ((b \otimes b') \otimes a )
= b \otimes (b' \otimes a)$
and then the universal maps
\[
\begin{cases}
\mu : \mu : \bB \otimes_\bB (\bB \otimes_\bA M)  \to (\bB \otimes_\bA M),
\quad
b \otimes (b' \otimes m) \mapsto (bb') \otimes m
\\
\nu :
\bB \otimes_\bB \bB  \to \bB,
\quad
b \otimes b' \mapsto (bb')
\end{cases}
\]
make the following diagram commute:
\[\begin{tikzcd}
	{(\bB \otimes_\bB \bB) \otimes_\bA M} & {\bB \otimes_\bB (\bB \otimes_\bA M)} \\
	& {\bB \otimes_\bA M}
	\arrow["\phi", from=1-1, to=1-2]
	\arrow["{\nu \otimes \mathrm{Id}_M}"', from=1-1, to=2-2]
	\arrow["\mu", from=1-2, to=2-2]
\end{tikzcd}\]
so if $\nu$ is an isomorphism, then $\mu$ is one too.\\
\end{proof}

\begin{thm}
If $g : \bA \to \bB$ is a $R$-algebra homomorphism and $\bB$ is a firm $\bB$-module, then there is a functor 
\[E_g : \,_\bA \qmod \to \,_\bB \fmod\]
such that 
\begin{itemize}
    \item $E_g(M) = \bB \otimes_\bA M$ where $\bB$ has the right $\bA$-module structure given by $g$
    \item For any $f : M \to N$, 
    \[
    E_g(f) : b \otimes m \mapsto b \otimes f(m)
    \]
\end{itemize}
\end{thm}

\begin{proof}
$\bullet$ For any $f : M \to N$, the map
\[\bB \times M \to N, 
\quad
(b, m) \mapsto b.f(m)\]
factors through $\bB \otimes_\bA M$ by the universal property of the tensor product, so $E_g(f)$ is a well-defined function. Moreover, $E_g(f)$ is clearly $R$-linear. Finally,
\[
E_g(f) (b. (b' \otimes m)) 
= E_g(f) ((bb') \otimes m)
=(bb') \otimes f(m) = b. (b' \otimes m)
\]

$\bullet$ For any $f : M \to N$ and $f' : N \to P$,
\[
E_g(f')E_g(f): b \otimes m \mapsto f'f(m)
\]
so $E_g(f')E_g(f) = E_g(f'f)$.
\end{proof}

\begin{dfnt}
The functor $E_g$ defined above is called the \textbf{extension of scalars along $g$}. 
\end{dfnt}

\begin{rem}
Usually, the reason we are interested in extension of scalars is that it is adjoint to restriction of scalars. 
However, here it is not necessarily the case unless we make additional assumptions on $g$. We will not need the adjunction in the general case in the sequel, so we leave the best characterization of the $g$ for which the extension of scalars is adjoint to the restriction of scalars functor for future work. 
\end{rem}

\subsection{Categories of $\mathbf{s}$-unital modules}

The algebras we will be working over have some additional properties that will allow us to describe firm modules in a more convenient way 
. A more detailed explanation of the most useful results can be found in \cite{thfmn} or \cite{tominaga,Komatsu}, here we only check that they hold under our definition of module over an algebra. All the following is folklore, that we prove only for completeness, as the literature about it is quite scattered.

\begin{dfnt}
A left (resp. right) module $M$ over an $R$-algebra $\bA$ is said to be $\mathbf{s}$\textbf{-unital} when:
\[
\forall m \in M, \; \exists e \in \bA, \quad
e.m = m
\]
(resp. $m.e = m$)
An $R$-algebra $\bA$ is said to be left (resp. right) $s$-unital when it is $s$-unital as a left (resp. right) $\bA$-module.

A bimodule is $s$-unital when it is left and right $s$-unital.
\end{dfnt}

\begin{rem}
Any left (or right) $s$-unital algebra is clearly an idempotent algebra.
\end{rem}

\begin{prop}
Any left (resp. right) $s$-unital module $M$ over a left (resp. right) $s$-unital $R$-algebra $\bA$ is firm.
\end{prop}

\begin{proof}
This is essentially the same proof as that given for rngs in \cite{thfmn}.

The main point is that, given some $m \in M$ and $e \in \bA$ such that $e.m = m$, $e \otimes m$ does not depend on the particular choice of $e$.

Indeed, take $e, e' \in \bA$ such that
\[e.m =  m = e'.m \]
By $s$-unitality of the algebra, there also exists $f \in \bA$ such  that
\[fe = e\]
\noindent and $f' \in \bA$ such that
\[f'(e' - fe') = e' - fe' \]
Then, by setting $g = f + f' - f'f$, we have:
\[
ge = fe + f'e - f'fe = e + f'(e - fe) = e
\]
\noindent and
\[
ge' = fe' + f'e' - f'fe' = fe' + e'-fe' = e'
\]
As such, $ge = e$ and $ge' = e'$, so:
\[
e \otimes m = ge \otimes m
= g \otimes em = g \otimes e' m
= ge' \otimes m = e' \otimes m
\]

It is then clear that the function that associates to each $m \in M$ the element $e \otimes m$, where $e \in \bA$ verifies $em = m$, is the inverse function of:
\[
\mu : \bA \otimes M \to M, \quad a \otimes m \mapsto am
\]
\noindent and thus $\mu$ is a bijective function, so it is an isomorphism.
\end{proof}

\begin{thm}
Let $\bA$ be a left (resp. right) $s$-unital algebra.

Then a left $\bA$-module is firm if and only if it is left (resp. right) $s$-unital.
\end{thm}

\begin{proof}
We already saw that left $s$-unital modules are firm.

Let us now take a firm left module $M$ over $\bA$.
Then for every $m \in M$ there exists $a \in \bA$ and $m' \in M$ so that:
\[m = a.m'\]
And then there exists $e \in \bA$ such that $ea = a$, thus:
\[e.m = e.(a.m') =(ea).m' = a.m' = m\]
so $M$ is left $s$-unital.
\end{proof}

\begin{cor}
\label{cor:10}
The category of left $s$-unital modules over a left $s$-unital ring is abelian.
\end{cor}

\section{A convenient abelian category of bimodules}

\label{sec:convenient}

When considering unital algebras $\bA$ and $\bB$, unital $\bA$-$\bB$-bimodules are the same as unital $\bA \otimes_{\mathbb{Z}} \bB^{\text{op}}$ modules. As such, unital categories of bimodules inherit all the interesting categorical structures of module categories. But in the nonunital case, this exact correspondence may fail: bimodules can always be seen as $\bA \otimes_{\mathbb{Z}} \bB^{\text{op}}$ modules (as we will show), but we do not expect the converse to be true. 


First of all, it is clear that $\bA \otimes_{\mathbb{Z}} \bB^{\text{op}}$ can be given a $R \otimes_{\mathbb{Z}} S^{op}$-algebra structure, when $\bA$ is a $R$-algebra, and $\bB$ is a $S$-algebra, by defining the product law by:
\[
(a \otimes b)(a' \otimes b') = (aa'\otimes b'b)
\]
and the bimodule law by:
\[
(r \otimes s).a \otimes b = (r.a \otimes s.b)
\]

Now, define a functor $\merg$ by the following:
\begin{itemize}
    \item[-] To every 
    $\bA$-$\bB$-bimodule $M$, for any $m \in M$, use the universal property of the tensor product to define a left $\bA \otimes_{\mathbb{Z}} \bB^{\text{op}}$-module law on $(M,+)$ such that 
    \[
    (a \otimes b).m = a.m.b 
    \]
    for $a \in \bA, b\in \bB$
    and 
    \[
    (r \otimes s).m = r.m.s
    \]
    for $s \in S, r\in R$
    and associate the resulting $\bA \otimes_{\mathbb{Z}} \bB^{\text{op}}$-module $\merg(M)$ to the bimodule $M$.
    \item[-] To every $\bA$-$\bB$-bimodule morphism $f: M \to N$, associate $f$ itself, since then 
    \[
    f((a \otimes b).m) = f(a.m.b) = a.f(m).b = (a \otimes b).f(m)
    \]
    and
    \[
    f((r \otimes s).m) = f(r.m.s) = r.f(m).s = (r \otimes s).f(m)
    \]
\end{itemize}
\begin{rem}
This functor is faithful, but it fails to be full as soon as there exists a nonzero left $\bA$-module $M$: the bimodule obtained by setting
\[\forall m \in M, \forall b \in \bB, \quad m.b = 0\]
is sent to a $\bA \otimes_{\mathbb{Z}} \bB^{\text{op}}$-module such that $\lambda.v = 0$ for any vector $v$ and any scalar $\lambda$.
\end{rem}

Since we do not have a general category equivalence, obtaining an abelian category of bimodules is a bit more complicated than in the unital case. This is the category of $s$-unital bimodules we are defining now: 


\begin{dfnt}[$s$-unital bimodules]
If $\bA$ is an algebra over $R$ and $\bB$ is an algebra over $S$, then a $\bA$-$\bB$-bimodule is said to be $s$-unital when it is left and right $s$-unital at the same time.

The $\Hat{\bA} \otimes \Hat{\bB}^{\text{op}} $-enriched category of $s$-unital bimodules is the full sub-$\Hat{\bA} \otimes \Hat{\bB}^{\text{op}} $-enriched-category of $\,_\bA \mathrm{Mod}_\bB$ made of the $s$-unital bimodules. We will denote it by $\,_\bA\smod_\bB$.
\end{dfnt}

\begin{rem}
The functor $\merg$ sends $s$-unital bimodules to $s$-unital modules. Its restriction
$\,_\bA \smod _\bB \to \,_{\bA \otimes_{\mathbb{Z}} \bB^{\text{op}}}{\smod}$ may 
sometimes be a category equivalence: when $\bA$ and $\bB$ are unital, $s$-unital bimodules are unital bimodules, so $\merg$ becomes the usual category isomorphism for unital bimodules. We believe that this functor is not full in the general case, although we failed to produce a counter-example.
\end{rem}

The category $\,_\bA\smod_\bB$ is the category we will be working on because of the following:

\begin{prop}
\label{prop:11}
$\,_\bA\smod_\bB$ is a $\Hat{\bA} \otimes \Hat{\bB}^{\text{op}} $-abelian category.
\end{prop}

\begin{proof}
Let $M$ and $N$ be $s$-unital $\bA$-$\bB$-bimodules.

Then they are $s$-unital as left $\bA$-modules. Take some $(m,n) \in M \times N$. There exists $a \in \bA$ such that $a.m = m$. Then, there exists $a' \in \bA$ such that $a'.(n-a.n) = n - a.n$. Finally, 
\[
(a+a'-a'a).m = a.m + a'.m - a'.(am) = a.m = m
\]
and
\[
(a+a'-a'a).n = a.n + a'(n-a.n) = a.n + n - a.n = n
\]
thus in the coproduct of bimodules $M \oplus N$ we have 
\[(a+a'-a'a).(m+n) = m + n\]
and $M \oplus N$ is left $s$-unital. 

In the same way, $M \oplus N$ is right $s$-unital, so it is a $s$-unital bimodule and $\,_\bA\smod_\bB$ admits the coproduct of any two objects.

The zero bimodule is clearly $s$-unital and serves as a zero object. Since $\,_\bA\smod_B$ has all finite coproducts and a zero object, it is an additive category.\\

If $M,N$ are objects in  $\,_\bA\smod_B$ and $f \in  \,_\bA\smod_B (M,N)$, then the kernel and cokernel abelian groups of the abelian group morphism $f: M \to N$ inherit the $s$-unital bimodule structure, and the corresponding kernel and cokernel abelian group morphisms are actually bimodule morphisms. As such, $f$ has both a kernel and a cokernel.\\

Since the canonical decomposition of any $f \in \,_\bA\smod_\bB (M,N)$ is exactly the canonical decomposition of the underlying abelian group morphism in the category of abelian groups, any $s$-unital bimodule morphism is strict.
\end{proof}

\section{Application: directed homology of  $(\infty,1)$-categories}

\label{sec:infinityone}

Here we present a homology theory for simplicially enriched categories, which constitute a model for $(\infty,1)$-categories in the case where the hom simplicial sets are Kan.

\begin{dfnt}[E.g. \cite{Lurie}]
A \textbf{simplicially enriched category} is a category enriched over the closed cartesian category of simplicial sets $\cS$, that is a category whose hom-sets are simplicial sets and so that, for every objects $A,B,C$, the composition
\[
\circ : \cS(A,B) \times \cS(B,C) \to \cS(A,C)
\]
is a simplicial map.

A simplicially enriched category is said to be \textbf{fibrant} if all the hom-sets are Kan complexes.
\end{dfnt}

Fibrant simplicially enriched categories are of particular interest because they model $(\infty,1)$-categories, \cite{Lurie}. However, the machinery developed here applies to general simplicially enriched categories. We believe this construction could be extended to quasi-categories, but this would require complicated technical lemmas that are unnecessary for our purposes.

\subsection{Defining the algebra of paths}
We want to embed a category's $1$-morphisms into an algebra that will be made to act on higher morphisms.
\begin{dfnt}
Let $\cS$ be a simplicially enriched category. For every object pair $(A, B)$, we can consider the set $\left( \cS(A,B) \right)_0$ of $0$-vertices of the simplicial set $\cS(A,B)$. Then we can define a category $\cS_0$ whose objects are those of $\cS$, whose morphisms $A \to B$ are the elements of $\left( \cS(A,B) \right)_0$, and whose composition 
    \[\circ_0: \cS_0(A,B) \times \cS_0(B,C) \to \cS_0 (A,C)\]
    is the restriction to $0$-vertices of the simplicial map 
    \[\circ : \cS(A,B) \times \cS(B,C) \to \cS(A,C)\]
    
This is the \textbf{category of $1$-morphisms} of $\cS$.
\end{dfnt}

\begin{dfnt}
Let $R$ be a commutative ring and let $\cS$ be a simplicially enriched category.

    Let $\bA$ be the free left-$R$-module generated by all morphisms in $\cS_0$. For every $f : X \to Y$ and $g : Z \to T$ morphisms in $\cS_0$, let 
    \[g \times f = \left\{
    \begin{array}{ll}
        g \circ_0 f & \mbox{ if } Z = Y \\
        0 & \mbox{otherwise}
    \end{array}
\right.\]

Then extending $\times$ to the entirety of $\bA$ by $R$-linearity gives rise to an $R$-algebra, called the \textbf{path algebra} (or convolution algebra, \cite{assocalg}) of $\cS$. It is denoted by $R[\cS]$
\end{dfnt}

In the general case, the path algebra does not have to be unital. Indeed, if $\cS$ has a finite collection of objects, we can see that the sum of all the identity morphisms (which is a finite sum) is an identity. There is no canonical unit in general for the path algebra of categories having an infinite number of objects. 
However, it has a lot of algebraic structure:

\begin{prop}
\label{prop:12}
For any simplicially enriched category $\cS$ and any commutative ring $R$, $R[\cS]$ is $s$-unital on both sides as a $R[\cS]$-bimodule (and thus idempotent).
\end{prop}

\begin{proof}
Take some $\sum_{i \in I} r_i . f_i \in R[\cS]$ with $I$ finite, $r_i \in R$ for all $i \in I$, and $f_i \in \cS_0(X_i, Y_i)$ for all $i$.

Let $X = \{X_i \; | \; i \in I \}$ 
and
$Y = \{Y_i \; | \; i \in I \}$. 
Then :

\begin{align*}
    \left(\sum_{A \in X} 1_A\right) \left(\sum_{i \in I} r_i. f_i\right) 
    &= \sum_{i \in I} r_i. (1_{X_i} f_i) + \sum_{i \in I} \sum_{A \in X, \; A \neq X_i} r_i.(1_A f_i)\\
    &= \sum_{i \in I} r_i. (1_{X_i} \circ_0 f_i) \\
    &= \sum_{i \in I} r_i.  f_i
\end{align*}
and similarly on the right.
\end{proof}

\begin{rem} The elements
of $\Hat{R}[\cS]$ have the form
\[ \left(\,\sum_{i \in I} \lambda_i t_i \, , \, r \,\right)
= \sum_{i \in I} \left( \lambda_i t_i \,,\, r \right) \]
where $(\lambda_i)_{i \in I}$ is a finite family of scalars in $R$, $(t_i)_{i \in I}$ is a finite family of $1$-morphisms and $r \in R$.

As such, $\Hat{R}[\cS]$ is generated by the elements of the form $(\lambda.t, r)$.
\end{rem}


In what follows, \textbf{let us take $\mathbf{R}$ and $\mathbf{R'}$ two fixed commutative (unital) rings} as well as $n$ a fixed integer.

\subsection{Defining chain complexes}
Consider some fixed simplicially enriched category $\cS$.

\begin{paragraph}{First step: defining a $R$-$R'$-bimodule of chains}
We remind the reader that, if $E$ is a set, then the free $R$-$R'$-bimodule generated by $E$ is the coproduct of the constant family of bimodules $\left( R \otimes_{\mathbb{Z}} R' \right)_{e \in E}$ where the left copy of $R$ is seen as a left $R$-module and the right copy of $R'$ is seen as a right $R'$-module.\\

Define $C_n \left( R, \cS, R'\right)$ as the free $R$-$R'$-bimodule generated by the set 
\[\bigcup_{A, B \text{ objects of } \cS} 
\left( \cS(A,B) \right)_n\]
of all $i$-simplices of any hom-object of $\cS$. 
For pairs $(A,B)$ of objects of $\cS$, we call $C_{n, A \to B} \left( R, \cS, R'\right)$ the $R$-$R'$-bimodule  generated by $\left( \cS(A,B) \right)_n$ and because of the definition above, $C_n \left( R, \cS, R'\right)$ is the coproduct of $C_{n, A \to B} \left( R, \cS, R'\right)$ over all pairs of objects $(A,B)$ of $\cS$. 

\end{paragraph}

\begin{paragraph}{Second step: defining the translation operators}

Take some $f \in \cS_0(A,B)$ and $g \in \cS_0(C,D)$ for some objects $A$, $B$, $C$ and $D$ of $\cS$. By a slight abuse of notation, let us denote by $(s_k)_{0 \leq k \leq n}$ the degeneracy operators $E_n \to E_{n+1}$ of any simplicial set $E$. Then 
\[s_{n-1}  \dots  s_0 (f) 
\in \left( \cS(A,B) \right)_n\]
and 
\[s_{n-1}  \dots  s_0 (g) 
\in \left( \cS(C,D) \right)_n\] 

Let us now define the left and right action (as $R$-$R'$)-bimodule endomorphisms) of elements $f \in \cS_0(A,B)$, respectively of $g \in \cS_0(C,D)$ on elements of the  $R$-$R'$-bimodule of $C_{n,X\rightarrow Y} \left( R, \cS, R'\right)$, for some $X$, $Y$ objects of $\cS$, by its values on basis elements $\sigma$. 

\begin{subparagraph}{First case}
$X \neq B$ or $Y \neq C$. Then we set $f \times \sigma \times g = 0$ for all $\sigma \in C_{n, X \to Y} \left( R, \cS, R'\right)$
\end{subparagraph}

\begin{subparagraph}{Second case}
$X = B$ and $Y = C$. 
 
Then for every $\sigma \in C_{n, X \to Y} \left( R, \cS, R'\right)$, we can use the restriction
\[\circ_n: (\cS(A,B))_n \times (\cS(B,C))_n \to (\cS(A,C))_n\]
of the simplicial map 
    \[\circ : \cS(A,B) \times \cS(B,C) \to \cS(A,C)\]
\noindent to define
\[f \times \sigma \times g = 
(s_{n-1} \dots  s_0) (f) \circ_n
\sigma \circ_n
(s_{n-1}  \dots  s_0) (g) 
\]
\end{subparagraph}

This extends to a $R$-$R'$-bimodule endomorphism of $C_n \left( R, \cS, R'\right)$
since this bimodule is free over the $n$ simplices of all hom-objects. 

\end{paragraph}

\begin{paragraph}{Third step: deriving the bimodule structure from translations}
Let us now consider arbitrary $\lambda , \mu$ $ \in R[\cS]$. Then we can write $\lambda = \sum_{ f \in \mathcal{F}} \lambda_f .f$ for some finite 
subset $\mathcal{F}$ of $\coprod\limits_{A,B \in \cS_0} \cS_0(A,B)$
and $\lambda_f \in R$ for all $f$, and this writing is unique when $\lambda_f \neq 0$ for all $f \in \mathcal{F}$. In the same way, we can write uniquely $\mu$ as $\sum_{ g \in \mathcal{G}} \mu_g .g $, with $\mathcal{G}$ a finite subset of $\coprod\limits_{C,D \in \cS_0} \cS_0(C,D)$. We now define $\lambda \times \sigma \times \mu$ for all $\sigma \in \bigoplus_{X,Y \in \cS_0}{C_{i,X\rightarrow Y} \left( R, \cS, R'\right)}$, by:
\[ \lambda \times \sigma \times \mu = \sum_{ f \in \mathcal{F}} \sum_{ g \in \mathcal{G}} f \times \left( \lambda_f. \sigma . \mu_g  \right) \times g \]

This is well-defined because of the unicity mentionned before.

Let us now show that this gives a $s$-unital $R[\cS]$-$R'[\cS]$-bimodule structure on $C_n \left( R, \cS, R'\right)$.

First, for all $\lambda$ and $\mu$, 
\[\sigma \mapsto \lambda \times \sigma \times \mu\]
is clearly $R$-$R'$-linear as a linear combination of linear functions. Moreover, for fixed $\sigma$, 
\[
(\lambda, \mu) \mapsto \lambda \times \sigma \times \mu
\]
is bilinear by definition.

By associativity of $\circ$, it is also clear that for any traces $f, f', g, g'$ and any $\sigma$, we have:
\begin{align*}
    f \times (f' \times \sigma) &= (f f') \times \sigma \\
   \text{and} \quad 
   (\sigma \times g') \times g &= \sigma \times (g'g) \\
   \text{and} \quad
   f \times (\sigma \times g) &= (f \times \sigma) \times g
\end{align*}

Finally, take $(f_i)_{i \in I}$  a finite family of $1$-morphisms so that $f_i \in \cS_0(X_i, Y_i)$. Take $(r_i)_{i \in I} $ a family of scalars in $R$ and $(s_i)_{i \in I}$ a family of scalars in $R'$. 
Define $e_X$ as the identity of $X$ in $\cS_0$ on any object $X$, and,
\[
E = \{e_{X_i} \; | \; i \in I\} \ \ \ \ \ 
F = \{e_{Y_i} \; | \; i \in I\}.
\]
%
Then
\begin{align*}
    \left(\sum_{e \in E} e\right) \left(\sum_{i \in I} r_i. f_i.s_i\right) 
    &= \sum_{i \in I} r_i. (e_{X_i} f_i).s_i + \sum_{i \in I} \sum_{e \in E, \; e \neq e_{X_i}} r_i.(e f_i).s_i\\
    &= \sum_{i \in I} r_i . (e_{X_i} f_i).s_i \\
    &= \sum_{i \in I} r_i.f_i.s_i
\end{align*}

and similarly
\[\left(\sum_{i \in I} r_i. f_i.s_i\right).\left(\sum_{e \in F} e\right) 
=\sum_{i \in I} r_i. f_i.s_i\]

Thus this is a left and right $s$-unital bimodule over a $s$-unital algebra on the left and a $s$-unital algebra on the right.

\end{paragraph}

All three steps put together result in the following theorem:

\begin{thm}
\label{thm:13}
There exists a  unique ($s$-unital) $R[\cS]$-$R'[\cS]$-bimodule structure on the $R$-$R'$-bimodule $C_n \left( R, \cS, R'\right)$ so that for all $1$-morphisms $f : X \rightarrow Y$, $g : U \rightarrow V$ and any $\sigma \in (\cS(A,B))_n$ we have

\[
f \times \sigma \times g = 
\begin{cases}
    
 0 & \text{ if } V \neq A \text{ or } B \neq X 
\text{, and } \\
(s_{n-1}\dots s_0)(f) \bullet_n 
\sigma
\bullet (s_{n-1}\dots s_0)(g) & \text{ otherwise}
\end{cases}
\]
\end{thm}

\begin{rem}
For a general $A$-$B$-bimodule $M$ over algebras and a predicate $\mathcal{P}$ on $M$, the statement
\[ 
\forall a \in A, \ \forall b \in B, \ \forall m \in M, \ \quad
\mathcal{P} (a.m.b)
\]
is not equivalent to the conjunction of the statements
\begin{align*}
    \forall a \in A, \ \forall m \in M, \quad
& \mathcal{P} (a.m)\\
 \forall b \in B, \ \forall m \in M, \quad
&\mathcal{P} (m.b)
\end{align*}

But here, the bimodule is left and right $s$-unital, so there is indeed an equivalence and we can write properties more concisely.
\end{rem}

\subsection{Defining the homology bimodules}
We now show that the construction of the singular homology groups of the hom sets behaves well in regards to the previous construction.

\begin{dfnt}{\textbf{Boundary operator}}

For any $n \geq 2$, define 
\[d_n: C_n \left( R, \cS, R'\right) \to C_{n-1} \left( R, \cS, R'\right)\]
as the unique $R$-$R'$-bimodule homomorphism so that, for any $A, B$ objects of $\cS$, for any $c \in (\cS(A,B))_n$ , $d_n c$ is the boundary of $c$ in the usual meaning of simplicial homology i.e. 
\[d_n c = \sum_{0 \leq j \leq n} {(-1)}^j \partial_i c \]
where $\partial_0, \dots, \partial_n : (\cS(A,B))_n \to (\cS(A,B))_{n-1}$ are the face operators of the simplicial set. 
\end{dfnt}

Thus, $d_n d_{n-1} = 0$ as in the classical case and we have a chain complex 
\[\begin{tikzcd}
	{...} & C_{n, A \to B} \left( R, \cS, R'\right) & C_{n-1, A \to B} \left( R, \cS, R'\right) & {...}
	\arrow["{d_{n+1}}", from=1-1, to=1-2]
	\arrow["{d_n}", from=1-2, to=1-3]
	\arrow["{d_{n-1}}", from=1-3, to=1-4]
\end{tikzcd}\]

\begin{prop}
For every $n \geq 2$, $d_n$ is a $R[\cS]$-$R'[\cS]$-bimodule homomorphism.
\end{prop}

\begin{proof}
By linearity of $d_n$, it suffices to show that, for $c \in (\cS (A,B))_n$, $f \in R[\cS]$ and $g \in R'[\cS]$,  $d_n \left( f \times c \times g \right)$ is the same as $f \times d_n c \times g$. We have:
\[
f \times d_n c \times g = 
f \times  \sum_{0 \leq j \leq n} {(-1)}^j  \partial_i c  \; \times g =
\sum_{0 \leq j \leq n} {(-1)}^j f \times \partial_i c  \times g\]
\noindent But for every $i$,
\begin{align*}
\partial_i s_n s_{n-1} \dots s_0 &=
(\partial_i s_n) s_{n-1} \dots s_0 \\
&= (s_{n-1} \partial_i) s_{n-1}  s\dots s_0\\
&= s_{n-1} (\partial_i s_{n-1})  s\dots s_0\\
&= \dots \\
&= s_{n-1} s_{n-2} \dots s_{i+1} (\partial_i s_{i+1}) s_i \dots s_0 \\
&= s_{n-1} s_{n-2} \dots s_{i+1} (s_i \partial i) s_i \dots s_0 \\
&= s_{n-1} s_{n-2} \dots s_{i+1} s_i (\partial i s_i) s_{i-1} \dots s_0 \\
&= s_{n-1} s_{n-2} \dots s_{i+1} s_i s_{i-1} \dots s_0
\end{align*}
\noindent so 
\begin{align*}
f \times \partial_i c  \times g &= 
(s_{n-1} \dots s_0)(f) \circ_n \partial_i c \circ_n (s_{n-1} \dots s_0)(g)\\
&= \partial_i (s_n \dots s_0)(f) \circ_n \partial_i c \circ_n \partial_i (s_n \dots s_0)(g) \\
&= \partial_i \left( 
 (s_n \dots s_0)(f) \circ_n  c \circ_n  (s_n \dots s_0)(g)
\right)\\
&= \partial_i ( f \times c \times g)
\end{align*}
Thus
\[
f \times d_n c \times g = \sum_{0 \leq j \leq n} {(-1)}^j \partial_i \left( f \times c \times g \right)  = d_n \left( f \times c \times g \right) \]
\end{proof}

\begin{dfnt}
We denote by $\G_n (R, \cS, R')$ the $n$-th homology group of the chain complex of $R(\cS)$-$R'(\cS)$-bimodules:
\[\begin{tikzcd}
	{...} & C_n \left( R, \cS, R'\right) & C_{n-1} \left( R, \cS, R'\right) & {...}
	\arrow["{d_{n+1}}", from=1-1, to=1-2]
	\arrow["{d_n}", from=1-2, to=1-3]
	\arrow["{d_{n-1}}", from=1-3, to=1-4]
\end{tikzcd}\]
\end{dfnt}

\subsection{Functoriality and invariance}
Let us take a fixed $n \in \mathbb{N}$ and $R$, $R'$ two fixed commutative rings.

Let $\cS$ and $\cT$ be two simplicially enriched categories and $f : \cS \to \cT$ a simplicially enriched functor. Then $f$ induces a functor $f_0 : \cS_0 \to \cT_0$, which in turn induces a $R$-$R'$-bimodule homomorphism
\[
f_* : C_n \left( R, \cS , R' \right) \to C_n \left( R, \cT , R' \right)
\]

The question we must ask is whether or not $f$ induces an algebra morphism $R[\cS] \to R[\cT]$ that would allow us to link the homology of $\cS$ to that of $\cT$ in some sense.\\

The natural way to construct such a morphism would be to define a $R$-algebra morphism $f_R$ from $R[\cS]$ to $R[\cT]$ that verifies:

\[f_R ( t ) = f_0 (t) \text{ for every } 1\text{-morphism } t\]

A $R$-linear function verifying this criteria always exists, but it may fail to be an algebra morphism. Indeed, if $\gamma$ is a $1$-morphism from $a$ to $b$ and $\delta$ is a $1$-morphism from $c$ to $d$, then in the case where $b \neq c$ the functor $f$ might send $b$ and $c$ to the same point. This would mean that $\gamma \times \delta = 0$ but $f_R (\gamma) \times f_R ( \delta) \neq 0$ potentially.\\

Just like in \cite{semi-ab} and \cite{eric}, this issue is solved is $f$ is supposed to be injective on objects, as $f_R$ is then really an $R$-algebra homomorphism. Moreover, we will show that this construction has enough functoriality properties to show that our homology is a dihomeomorphism  invariant.

\begin{thm}
\label{thm:15}
Let $f : \cS \to \cT$ be a simplicially enriched functor  and $n \in \mathbb{N}$. 
Then $f$ induces a functorial $R$-$R'$-bimodule morphism \[f_* : C_n(R, \cS, R') \to C_n(R, \cT, R')\]
Moreover, if $f$ is injective on objects, then
\begin{enumerate}
    \item[1)] $f$ induces algebra morphisms $f_R : R[\cS] \to R[\cT]$ and $f_S : R'[\cS] \to R'[\cT]$
    \item[2)] For any $a \in R[\cS]$, $b \in R'[\cS]$ and $m \in C_n \left( R, \cS , R' \right)$,
    \begin{equation*}
    \left\{
    \begin{aligned}
      f_* (a.m) &= f_R (a). f_* (m) \\
      f_* (m.b) &=  f_* (m) . f_S(b)
    \end{aligned}
    \right.
    \end{equation*}
    \item[3)] $f_R$ and $f_S$ are functorial
\end{enumerate}
\end{thm}

The second point could have been equivalently written as
\[f_* (a.m.b) = f_R (a). f_* (m) . f_S(b)\]
because of the $s$-unitality. However, one should always be careful with such simplifications when working over nonunital rings. 

\begin{proof}
We already explained how to define $f_*$, and it is clear from the definition that it will be functorial. All that remains is to prove the additional functoriality properties in the case where $f$ is injective on objects.

Suppose that $f$ is injective on objects. As explained before, there already exists a unique $R$-algebra morphism $f_R : R[T_{\cS}] \to R[T_{\cT}]$ that verifies
\[f_R(t) = f(t) \text{ for every } 1 \text{-morphism } t\]
If $t: x \to y$ and $t': a \to b$ are $1$-morphisms in $\cS$, then there are two cases to consider for $t \times t'$:
\begin{itemize}
    \item Case 1: $y = a$, so $f(y) = f(a)$ and thus
    \[
        f_R(t) \times f_R(t') = f(t) \circ f(t') 
        = f (t  t') 
        = f_R( tt')
    \]
    \item Case 2: $y \neq a$, in that case $f(y) \neq f(a)$ by injectivity of $f$ and thus
    \[
        f_R(t) \times f_R(t') = 0 
        = f_R(0)
        = f_R( tt')
    \]
\end{itemize}
 
By $R$-linearity, $f_R : R[\cS] \to R[\cT]$ is thus a $R$-algebra homomorphism. $f_S$ is defined in the same manner.\\

Now let us prove \textit{2)}.
Since 
\[f_* (a.(m_1 + m_2)) = f_*(a.m_1) + f_*(a.m_2),
\quad
f_* ((m_1 + m_2).b) = f_*(m_1.b) + f_*(m_2.b)\]
and 
\[f_R(a).f_*(m_1 + m_2) = 
f_R(a).f_*(m_1) + f_R(a).f_*(m_2),
\]
\[
f_*(m_1 + m_2).f_S(b) = 
f_*(m_1).f_S(b) + f_*(m_2).f_S(b)\]
we can restrict ourselves to the case where $m$ is a simplex. As such, let us take objects $A,B$ of $\cS$ and assume $m \in (\cS(A,B))_n$.

Since the bimodule law is $R[\cS]$-$R[\cT]$-bilinear, we only have to prove the relation for generating elements of  $R[\cS]$ and $R[\cT]$. We can thus restrict ourselves to the case where $a$ and $b$  are $1$-morphisms.

As such, let us take $U, V$ objects of $\cS$ and $t, t' \in (\cS(U,V))_0$. We need to show that $f_*(t.m) = f_0(t).f_*(m)$ and $f_*(m.t') = f_*(m).f_0(t')$. Both relations are proved in the exact same way, so we will only prove the first one.

There are two cases to consider:
\begin{itemize}
    \item {Case 1:} $V \neq A$
    
    In that case $a.m = 0$ so $f_*(t.m) = 0$, and also $f(V) \neq f(A)$ by injectivity of $f$ so 
    \[ f_0 (t).f_*(m) = 0 = f_*(t.m)\]
    
    \item {Case 2:}
    $V = A$, then 
    \begin{align*} 
f_* (t.m) &= 
f_n ( s_0 \dots s_{n-1} (t) \circ_n m)
\\
&=  f_n ( s_0 \dots s_{n-1} (t)) \, \circ_n \, f_n(m)\\
&=  s_0 \dots s_{n-1} (f_0(t)) \, \circ_n \, f_*(m) \\
&= f_0(t). f_*(m) 
\end{align*}
\end{itemize}
The proof of \textit{3)} follows directly from the fact that $(gf)_* = g_* f_*$.
\end{proof}

\begin{prop}
Let $f : \cS \to \cT$ be a simplicially enriched functor and $n \in \mathbb{N}$.

Then the following diagram commutes for every $a,b \in X$:
\[\begin{tikzcd}
	C_{n+1, a \to b}\left(R, \cS, R'\right) 
	& C_{n+1, f(a) \to f(b)}\left(R, \cT, R'\right) \\
	C_{n, a \to b}\left(R, \cS, R'\right) 
	& C_{n, f(a) \to f(b)} \left(R, \cT, R'\right)
	\arrow["{f_*}", from=1-1, to=1-2]
	\arrow["\partial"', from=1-1, to=2-1]
	\arrow["\partial", from=1-2, to=2-2]
	\arrow["{f_*}"', from=2-1, to=2-2]
\end{tikzcd}\]

\end{prop}

\begin{proof}
This follows from the same straightforward computation as in standard singular homology.
\end{proof}

\begin{cor}
\label{cor:17}
\begin{enumerate}
    \item If $f : \cS \to \cT$ is a simplicially enriched functor then, for every $n \in \mathbb{N}$, $f_*$ induces a $R$-$R'$-bimodule homomorphism 
    \[\G_n(R, f, R') : \G_n(R, \cS, R') \to \G_n(R, \cT, R')\]
    
    When there is no ambiguity, we will simply denote it by $\G_n(f)$
    \item If $f$ is injective on objects, then for every $a \in R(\cS)$, $b \in R'(\cT)$ and $u \in G_n(R, \cS, R')$:
    \[ \G_n(f)(a.u.b) = f_R(a).\,\G_n(f)(u)\,.f_S(b)\]

\end{enumerate}
\end{cor}

\begin{cor}
\label{cor:18}
Homology is invariant up to isomorphisms of simplicially enriched categories i.e. isomorphic simplicially enriched categories give rise to homology bimodules that are isomorphic in the category of bimodules over $R$-algebras on the left and $R'$-algebras on the right.
\end{cor}

This follows from the fact that isomorphisms are injective functions, and as such induce functorial homomorphisms in homology. 

\subsection{Relative homology for $(\infty,1)$-categories}

This subsection generalizes an earlier construction of directed topology for finite precubical sets \cite{eric}.

\paragraph{Overview } A crucial aspect of singular homology is that the chain complex $C_n(X)$ of a space $X$ can be quotiented out by the chain complex $C_n(A)$ of a subspace $A$ in the abelian category  $_R \md$. Then the boundary operator 
\[\partial : C_n(X) \to C_{n-1}(X)\]
induces a boundary operator
\[ \partial : C_n(X){/ C_n(A)} \to C_{n-1}(X){/ C_{n-1}(A)}\]
Which gives rise to a short exact sequence of chain complexes over $_R \md$. By a standard homological algebra result, this short exact sequence of chain complexes induces a long exact sequence relating the homology groups of each complex.\\

It would be very desirable to have a similar result in our case that relates the homology of an $(\infty,1)$-category to that of a sub-$(\infty,1)$-category. However, since the category of bimodules over varying algebras is not abelian we must be careful.

\subsubsection{Exact sequence of a sub-simplicially enriched category}
\label{sec:exact}

Take $\cS$ a simplicially enriched category and $\cT$ a simplicially enriched subcategory of $\cS$. \\

Let $n \in \mathbb{N}$. Then $(\cT(A,B))_n \subset (\cS(A,B))_n$, so $C_n(R, \cT, R')$ can be seen as a sub-$R$-$R'$-bimodule of $C_n(R, \cS, R')$.

\begin{dfnt}

\begin{enumerate}
    \item The \textbf{bimodule of extended $\cT$-chains} is the sub-$R[\cS]$-$R'[\cS]$-bimodule of $C_n(R, \cS, R')$ generated by $C_n(R, \cT, R')$ (or, equivalently, by the n-simplices of $\cT$). Denote it by:
    \[C^{\cS}_n(R, \cT, R')\]
 
    \item The \textbf{bimodule of  $\cT$-relative $\cS$-chains} is 
    \[C_n(R, \cS / \cT, R') := C_n(R, \cS, R'){/ C^{\cS}_n(R, \cT, R')}\]
\end{enumerate}
\end{dfnt}
 
 \begin{prop}
In the abelian category of $s$-unital $R[\cS]-R'[\cS]$-bimodules, there is a short exact sequence of chain complexes:

 \[\begin{tikzcd}
	0 & C^{\cS}_*(R, \cT, R') & C_*(R, \cS, R') & C_*(R, \cS/ \cT, R') & 0
	\arrow[from=1-1, to=1-2]
	\arrow["j", from=1-2, to=1-3]
	\arrow["p", from=1-3, to=1-4]
	\arrow[from=1-4, to=1-5]
\end{tikzcd}\]
 \end{prop}

\begin{proof}
Since the boundary operator of $\cS$ is a $R[\cS]$-$R'[\cS]$-bimodule morphism, it sends the bimodule generated by $C_n(R,\cT, R')$ into the bimodule generated by 
\[\partial \left( C_n(R,\cT, R')\right) \subset C_{n-1}(R,\cT, R')\]
thus 
\[\partial \left(C^{\cS}_n(R,\cT, R')\right) \subset
C^{\cS}_{n-1}(R,\cT, R')\]
and $\partial$ induces a boundary operator on 
$\left(
C^{\cS}_n(R,\cT, R')
\right)_{n \in \mathbb{N}}$
This gives rise to the chain complex $C^{\cS}_*(R, \cT, R')$, and the inclusion $j_n: \ C^{\cS}_n(R,\cT, R') \to C_n(R, \cS, R')$ is by definition a monomorphism of chain complexes.\\

Moreover, the fact that
\[\partial \left(C^{\cS}_n(R,\cT, R')\right) \subset
C^{\cS}_{n-1}(R,\cT, R')\]
also means that the boundary operator on $\cS$ induces a boundary operator:
\[
\partial: C_n(R, \cS/ \cT, R') \to C_{n-1}(R, \cS/ \cT, R')
\]
that defines the chain complex $C_*(R, \cS/ \cT, R')$, and 
$p: \ C_*(R, \cS, R') \to C_*(R, \cS/ \cT, R')$, the canonical projection, 
is by definition a chain complex epimorphism, whose kernel is indeed the image of the monomorphism $j=(j_n)_{n\in \N}$. 
\end{proof}

\begin{dfnt}
\begin{itemize}
    \item The homology of the chain complex $C_*(R, \cS/ \cT, R')$ is called the \textbf{homology of $\cS$ relative to $\cT$} denoted by $\G (R, \cS/ \cT, R')$.
    \item The homology of the chain complex $C^{\cS}_*(R,  \cT, R')$ is called the \textbf{homology of $\cT$ extended in $\cS$} denoted by $\G^{\cS} (R, \cT, R')$.
\end{itemize}
\end{dfnt}

According to 
the zig-zag lemma, the short exact sequence of chain complexes yields a long exact sequence of homology objects in the abelian category of $s$-unital bimodules:

\[\begin{tikzcd}
	&& \vdots \\
	\G^{\cS}_{n+1}(  \cT) && \G_{n+1}(\cS) &&  \G_{n+1}(\cS/ \cT) \\
	\G^{\cS}_n( \cT) && \G_{n}( \cS) &&  \G_n(\cS/ \cT)\\
	\G^{\cS}_{n-1}(  \cT) && \G_{n-1}(\cS) &&  \G_{n-1}(\cS/ \cT) \\
	&& \vdots & {}
	\arrow[from=1-3, to=2-1]
	\arrow["{j_*}", from=2-1, to=2-3]
	\arrow["{p_*}", from=2-3, to=2-5]
	\arrow[dashed, "{\delta_{n+1}}", from=2-5, to=3-1]
	\arrow["{j_*}"', from=3-1, to=3-3]
	\arrow["{p_*}"', from=3-3, to=3-5]
	\arrow[dashed, "{\delta_{n}}", from=3-5, to=4-1]
	\arrow["{j_*}"', from=4-1, to=4-3]
	\arrow["{p_*}"', from=4-3, to=4-5]
	\arrow[from=4-5, to=5-3]
\end{tikzcd}\]

As such, the homology of $\cS$ can be computed from the homology of $\cT$ extended in $\cS$ and the homology of $\cS$ relative to $\cT$.

\subsubsection{Formulation using the extension of scalars functor}
We would like to have a systematic way to compute the chain complex $C^{\cS}_*(R,\cT, R')$, or at least its homology bimodules, from $C_*(R,\cT, R')$ and $R[\cS]$.
A partial answer can be given by describing $C^{\cS}_n(R,\cT, R')$ in terms of an extension of scalars.

Let $R$ and $R'$ be commutative rings, and $n$ an integer. 
Let $\cS$ be a simplicially enriched category and $\cT$ a sub-simplicially enriched category of $\cS$. 

Then the canonical injection functor $\cT \to \cS$ induces a canonical functor of simplicially enriched categories $\cT\to \cS$. As an injective dimap, the inclusion $\cT\to \cS$ induces natural algebra injections $R[\cT] \to R[\cS]$ and $R'[\cT] \to R'[\cS]$.
Then the extension of scalars along those inclusions of algebras gives us a $s$-unital $R[\cS]$-$R'[\cS]$-bimodule:

\[ R[\cS] \otimes_{R[\cT]} C_n(R, \cT, R') \otimes_{R'[\cT]} R'[\cS]\]

Since $f: R[\cS] \times C_n(R, \cT, R') \times R'[\cS] \to C_n(R, \cS, R')$ is $\mathbb{Z}$-trilinear and balanced on each side, it factorizes through the tensor product into an application:

\[\iota_n:  \begin{cases} 
      R[\cS] \otimes_{R[\cT]} C_n(R, \cT, R') \otimes_{R'[\cT]} R'[\cS]
      &\to C_n(R, \cS, R')\\
      x \otimes m \otimes y &\mapsto x.m.y
   \end{cases}
\]

\begin{dfnt}
\label{def:transfer}
We will call $\iota_n$ the \textbf{transfer map}. 
\end{dfnt}
The transfer map is always a $R[\cS]$-$R'[\cS]$-bimodule morphism, but may not be an injection. It is useful because of the following obvious property:

\begin{prop}
The image of the transfer map is exactly $C^{\cS}_n(R,\cT, R')$. 
\end{prop}

Now, the boundary operators $\partial_*$ of the complex $C_{*} (\cT)$ induce operators
\[
\mathrm{Id}_{R[\cS]} \otimes \partial_n \otimes \mathrm{Id}_{R'[\cS]}
: \;
R[\cS] \otimes C_n(\cT) \otimes R'[\cS]
\to
R[\cS] \otimes C_{n-1}(\cT) \otimes R'[\cS]
\]
\noindent that define a chain complex
\[
\left(
R[\cS] \otimes C_{*}(\cT) \otimes R'[\cS]
\; , \;
\mathrm{Id}_{R[\cS]} \otimes \partial_{*} \otimes \mathrm{Id}_{R'[\cS]}
\right)
\]

Since the transfer map is a morphism of $R[\cS]-R'[\cS]$-bimodules, it is also a morphism of $R[\cT]-R'[\cT]$-bimodules by restriction of scalars, and thus it induces a morphism of chain complexes:  
\[\iota : 
R[\cS] \otimes C_{*}(\cT) \otimes R'[\cS]  
\to
C_{*}^{\cS}(\cT)
\]

\begin{prop}
The transfer map induces a short exact sequence of chain complexes of $s$-unital $R[\cS]-R'[\cS]$-bimodules

 \[\begin{tikzcd}
	R[\cS] \otimes C_{*}(\cT) \otimes R'[\cS]  & C_{*}^{\cS}(\cT) & 0
	\arrow["\iota", from=1-1, to=1-2]
	\arrow[ from=1-2, to=1-3]
\end{tikzcd}\]

As such, there is a long exact sequence of chain complexes:
\[\begin{tikzcd}
	0 & {\mathrm{Ker}( \iota)} & {R[\cS] \otimes C_{*}(\cT) \otimes R'[\cS]} \\
	& {C_{*}^{\cS}(\cT) } & {C_*(R, \cS, R')} & {C_*(R, \cS/ \cT, R')} & 0
	\arrow[from=1-1, to=1-2]
	\arrow[from=1-2, to=1-3]
	\arrow["\iota", from=1-3, to=2-2]
	\arrow[from=2-2, to=2-3]
	\arrow[from=2-3, to=2-4]
	\arrow[from=2-4, to=2-5]
\end{tikzcd}\]
\end{prop}
We will show in next section that, for directed spaces, the case where $ \mathrm{Ker}( \iota) = 0$ has a nice geometric interpretation. In this case of interest, the long exact sequence completely collapses and thus gives us a method to concretely compute the homology of $\cS$ from that of $\cT$. This allows us to reduce the complexity of the problem. 

However, it is still an open question whether this sequence can be used for computations in a more general setting through spectral sequences (assuming that $\mathrm{Ker}( \iota)$ is well understood for the chosen $\cS$ and $\cT$).

\section{Application: directed homology of directed spaces}

We will use the construction of the previous section and apply it to the theory of d-spaces. The motivation for this is to generalize homological results about precubical spaces \cite{eric}, which act as a natural combinatorial model for directed spaces.

\subsection{Directed spaces}

There exists multiple notions of space with a direction of time, but one of them stands out for being particularly general and well-fitted for homology: directed spaces (see e.g. \cite{thebook,grandisbook} for all definitions in this section), which we present briefly in this section.

\begin{dfnt}[Directed space, \cite{grandisbook}]
A \textbf{directed space} is a pair $\Vec{X} = (X, dX)$ where $X$ is a topological space and $dX$ is a set of continuous paths in $X$, called the \textbf{directed paths} 
of the space, such that;
\begin{enumerate}
    \item[(i)] Every constant path is a directed path
    \item[(ii)] Every increasing reparametrization of a directed path is a directed path i.e. for every directed path $\gamma \in dX$ and for any function $\varphi : [0,1] \to [0,1]$ which is continuous and increasing, $\gamma \circ \varphi \in dX$  
    \item[(iii)] For any directed paths $\gamma$ and $\delta$ so that $\gamma(1) = \delta(0)$, the concatenation $\gamma \cdot \delta$ of $\gamma$ and $\delta$ is a directed path
\end{enumerate}
\end{dfnt}

One simple example of a directed space is the closed interval $[0,1]$ with increasing function $[0,1] \to [0,1]$ as directed paths. This space is the \textbf{directed unit interval}.

Morphisms of directed spaces are defined as follows.
\begin{dfnt}
Let $\Vec{X} = (X, dX)$ and $\Vec{Y} = (Y, dY)$ be two directed spaces. A continuous function $f : X \to Y$ is called a \textbf{directed map}, or \textbf{dimap}, when:
\[\forall \gamma \in dX, \quad 
f \circ \gamma \in dY\]

\end{dfnt}
A dimap from the directed unit interval to a directed space $\Vec{X}$ is called a \textbf{dipath}. The dipaths of $\Vec{X}$ are exactly its directed paths.

Directed spaces together with dimaps (and composition of functions) form the category of directed spaces. An isomorphism in this category is called a \textbf{dihomeomorphism}. 

In a given directed space, the set of paths from one point to another is endowed with the compact-open topology.

\begin{dfnt}
Let $\Vec{X} = (X, dX)$ be a directed space and $x,y \in X$. 

The set of all directed paths $\gamma \in dX$ so that 
\[\gamma(0) = x \textit{ and } \gamma (1) = y\]
together with the compact-open topology is called the \textbf{path space} from $x$ to $y$ and is denoted by $\mathbf{P(\Vec{X})_x^y}$ 
\end{dfnt}

Given a directed space whose points model the states of a program and whose directed paths model all the possible sequences of state transitions, the execution traces should be understood as paths \textit{up to bijective increasing reparametrization} \cite{grandisbook} in the sense that the speed with which we transition through the states does not matter. 

The relation $\sim$ ``being the same up to bijective increasing reparametrization'' defined by
$\gamma \sim \delta$  iff there exists a continuous increasing bijective function 
\[f : [0,1] \to [0,1]\]
\noindent such that
\[\delta = \gamma \circ f \]
is indeed an equivalence relation.

\begin{dfnt}
Given a directed space $\Vec{X} = (X, dX)$, a \textbf{trace} from $x \in X$ to $y \in X$ is an equivalence class of paths from $x$ to $y$ for the equivalence relation here denoted by $\sim$.
If $\gamma$ is a dipath, we will denote by $\mathbf{[\gamma]}$ the corresponding trace.
\end{dfnt}

Now note that, if $\gamma$ and $\delta$ are dipaths such that $\gamma(1) = \delta(0)$, and if $f, g : [0,1] \to [0,1]$ are continuous, bijective and increasing functions, then:
\begin{enumerate}
    \item $\gamma \circ f (1) = \delta \circ g (0)$
    \item $(\gamma \circ f) \cdot  (\delta \circ g) = (\gamma \cdot \delta) \circ (f \cdot g)$
\end{enumerate}

This allows us to define concatenation on traces, and this gives rise to a category.
\begin{dfnt}[Trace category \cite{thebook}]
Let $\Vec{X} = (X, dX)$ be a directed space. Its \textbf{trace category} is the category with objects the points of $X$, with morphisms from $x$ to $y$ the traces from $x$ to $y$, and with composition given by concatenation of traces.
\end{dfnt}

\subsection{Homology of directed spaces}

In this section, we recap for completeness, the construction of Section 6 of \cite{semi-ab}. 


We recall the following: 
the standard simplex of dimension $n$ is $$
\Delta_n=\left\{(t_0,\ldots,t_n) \mid \forall i\in \{0,\ldots,n\}, \ t_i \geq 0 \mbox{ and } \sum\limits_{j=0}^n t_j=1\right\}
$$
For $n \in \N$, $n
\geq 1$ and $0 \leq k \leq n$, the $k$th ($n-1$)-face (inclusion) of the topological $n$-simplex is the subspace inclusion
$$\delta_k: \ \Delta_{n-1} \rightarrow \Delta_n$$
induced by the inclusion
$$(t_0,\ldots,t_{n-1}) \rightarrow (t_0,\ldots,t_{k-1},0,t_k,
\ldots, t_{n-1})$$
For $n \in \N$ and $0\leq k < n$, the 
$k$th degenerate 
$n$-simplex is the surjective map
$$
\sigma_k: \ \Delta_n \rightarrow \Delta_{n-1}$$
\noindent induced by the surjection: 
$$(t_0,\ldots,t_n)\rightarrow (t_0,\ldots,t_{k}+t_{k+1},\ldots,t_n)$$

Still as in \cite{semi-ab},  
we call $p$ an $i$-trace, $i\geq 1$, or trace of dimension $i$ of $X$, any 
 continuous map 
$$p: \Delta_{i-1} \rightarrow Tr(X)$$ 
\noindent which is such that: 
\begin{itemize}
\item $p(t_0,\ldots,t_{i-1})(0)$ does not depend on $t_0,\ldots,t_{i-1}$ and we denote it by $s_p$ (``start of $p$")
\item and $p(t_0,\ldots,t_{i-1})(1)$ is constant as well, that we write as $t_p$ (``target of $p$")
\end{itemize}
We write $T_i(X)$ for the set of $i$-traces in $X$. We write $T_i(X)(a,b)$, $a$, $b \in X$, for the subset of $T_i(X)$ made of $i$-traces from $a$ to $b$. 

In short, an $i$-trace of $X$ is a particular $(i-1)$-simplex of the trace space of $X$, for which we can define boundary and degeneracy maps, as usual: 

\begin{itemize}
\item Maps $d_{j}$, $j=0,\ldots,i$ acting on $(i+1)$-paths 
$p: \ \Delta_i \rightarrow Tr(X)$, $i\geq 1$: 
$$d_j(p)=p\circ \delta_j$$
\item Maps $s_k$, $k=0,\ldots,i-1$ acting on $i$-paths 
$p: \Delta_{i-1} \rightarrow Tr(X)$, $i\geq 1$: 
$$s_k(p)=p\circ \sigma_k$$
\end{itemize}

This makes $\K_{\Vec{X}}(a,b)=((T_i(X)(a,b))_{i\geq 0},d_j,s_k)$ (for some $a$, $b$ in $X$), and $\K_{\Vec{X}}=((T_i(X))_{i\geq 0},d_j,s_k)$ simplicial sets as shown in \cite{semi-ab}. 

Now, we can define a composition $q\circ p$ of an $i$-trace $p$ with an $j$-trace $q$, derived from the concatenation operation on traces $*$: $q\circ p$ is a $(i+j)$-trace with $q\circ p(t_0,\ldots,t_{i-1},t_{i},\ldots,t_{i+j-1})=p(t_0,\ldots,t_{i-1})*q(t_i,\ldots,t_{i+j-1})$.

\begin{thm}
\label{thm:21}
With the construction above, $\K_{\Vec{X}} (x,y)$ is a simplicially-enriched category.
\end{thm}

\begin{proof}
The composition of simplexes is associative, as concatenation of traces is known to be associative (this is concatenation of paths, modulo reparametrization). 
\end{proof}

\begin{rem}
This simplicially enriched category is fibrant, the proof being essentially the same as that for the singular simplicial complex in undirected topology.
\end{rem}

\begin{dfnt}
For $R$ and $R'$ two commutative unital rings, the homology of a directed space $\Vec{X}$ is defined as 
\[\G_n (R, \Vec{X}, R')  := \G_n (R, \K_{\Vec{X}}, R')\]
\end{dfnt}


\begin{rem}
A similar construction has also been considered in the context of framed bicategories, see \cite{augustin2026}.
\end{rem}


\subsection{Relative homology for nice subspaces}
Let $\Vec{X} = (X, dX)$ be a directed space.

\begin{dfnt}
A \textbf{subspace} of $\Vec{X}$ is a directed space $\Vec{A} = (A, dA)$ so that
\begin{enumerate}
    \item $A \subset B$
    \item $dA \subset dB$
\end{enumerate}
\end{dfnt}

This definition of a subspace is coherent with the category-theoric notion of subobject in directed spaces.

Given a subspace $\Vec{A}$ of $\Vec{X}$, $\K_{\Vec{A}}$ is a sub-simplicially enriched category of $\K_{\Vec{X}}$, and we can try to understand the long exact sequence of relative homology described in the last section. We will study a particular case where the homology of $\Vec{A}$ relative to $\Vec{X}$ can be expressed explicitly using the transfer map.

\begin{dfnt}
Let $\Vec{A}$ be a subspace of a directed space.
\begin{itemize}
    \item $\Vec{A}$ is said to be \textbf{full} if for any directed path $\gamma \in dX$ such that $\gamma([0,1]) \in A$, $\gamma \in dA$
    \item $\Vec{A}$ is said to be \textbf{separating} if for any directed path $\gamma \in dX$, $\gamma^{-1} (A)$ is connected, and thus an interval of $[0,1]$ (which may be the empty interval).
\end{itemize}

\end{dfnt}

Separatingness amounts to saying that any directed path in $\Vec{X}$ enters and leaves $A$ one time at most. This property has been used in \cite{eric} for getting relative homology exact sequences in the case of directed homology of precubical sets, and we follow here the same path towards finding such homology exact sequences. 



\begin{dfnt}
We will call a pair of directed spaces $(\Vec{A}, \Vec{X})$ a \textbf{relative pair} when
\begin{enumerate}
    \item $\Vec{A}=(A,dA)$ is a full subspace of $\Vec{X}=(X,dX)$
    \item $\Vec{A}$ is separating
    \item $A$ is closed in $X$
\end{enumerate}
\end{dfnt}

.

\subsection{Canonical injection of chains}
 What we will now show is that, for a relative pair, the transfer map of Definition \ref{sec:extensionscalar} from the extension of scalars of the chain complex on $\Vec{A}$ (from $R[\K_{\Vec{A}}]$-bimodules to $R[\K_{\Vec{X}}]$-bimodules) to the chain complex on $\Vec{X}$ is indeed an injective function. To prove this, we will need a notion of reduced forms.

Similarly to \cite{eric}, where the author used a notion of canonical form for elements of the chain complex of bimodules, in order to understand relative (directed) homology of certain precubical sets, we define a reduced form below, in the case of directed spaces:  

\begin{dfnt}[Reduced form]
An element of the extension of scalars of $C_n(\K_{\Vec{A}})$ from a $R[\K_{\Vec{A}}]$-bimodule to a $R[\K_{\Vec{X}}]$-bimodule: 
\[\sum_{i \in I} a_i \otimes m_i \otimes b_i 
\;
\in
\;
R[\K_{\Vec{X}}] \otimes C_n(\K_{\Vec{A}}) \otimes R'[\K_{\Vec{X}}]\]
is said to be in \textbf{reduced form} if 
\begin{enumerate}
    \item For every $i \in I$, $a_i = \lambda_i u_i$ where $\lambda_i \in R$ and $u_i$ is a trace that cannot be factored by a non-constant trace of $A$ on the right, 
    i.e. if $u_i = t t'$ with $t'$ a trace in $A$, then $t'$ is a constant trace
    \item For every $i \in I$, $b_i = \mu_i v_i$ where $\mu_i \in R'$ and $v_i$ is a trace that cannot be factored by a non-constant trace of $A$ on the left, 
    i.e. if $v_i = t' t$ with $t'$ a trace in $A$, then $t'$ is a constant trace
    \item For every $i \in I$, $m_i$ is a $n$-simplex in the trace space of $\Vec{A}$
    
\end{enumerate}
\end{dfnt}

Reduced forms will allow us to understand better the kernel of the transfer map for relative pairs.

\begin{prop}
If $(\Vec{A}, \Vec{X})$ is a relative pair then every element has a reduced form.
\end{prop}

\begin{proof}
Every element of $R[\K_{\Vec{X}}] \otimes C_n(\K_{\Vec{A}}) \otimes R'[\K_{\Vec{X}}]$ is a linear combination of elementary tensors, so we only have to prove the decomposition for elementary tensors.

By trilinearity of $(u, x, v) \mapsto u \otimes x \otimes v$, we can reduce the problem into finding a reduced form of $t \otimes m \otimes t'$ where $m$ is a trace simplex of $\Vec{A}$ and $t$, $t'$ are traces in $\Vec{X}$ such that $t \otimes m \otimes t' \neq 0$. 

Let us  take a path $\gamma \in dX$ associated to a trace $t = [\gamma]$. Then $\gamma^{-1}(A)$ is a closed interval.

 Moreover, $t \otimes m \otimes t' \neq 0$ so $t .m.t' \neq 0$. Indeed, if $t.m.t'$ was null, then we would have either $t.m = 0$ or $m.t' = 0$. Consider the first case for example: denoting by $c$ the constant trace at the beginning endpoint of $m$, we would have $t.c = 0$ and thus:
 \[
 t \otimes m \otimes t' = 
 t \otimes c.m \otimes t' =
 t.c \otimes m \otimes t' = 0
 \]
 \noindent which is absurd.
 
 Thus, $\gamma(1) = m(\cdot)(0) \in A$. Since $\gamma^{-1}(A)$ is closed, and by closeness of $A$, the connected component of $1$ in $\gamma^{-1}(A)$ is a compact interval, and we will call its smallest point $s_0$. Splitting $\gamma$ at $s_0$ gives two new traces $t_1, t_2$ such that:
\begin{itemize}
    \item $t = t_1 t_2$
    \item $t_2$ is a trace in $\Vec{A}$ (because $\Vec{A}$ is closed) 
    \item $t_1$ cannot be factored on the right by a non-constant trace of $\Vec{A}$
\end{itemize}

We can write in the same way $t' = t'_1 t'_2$, where $t'_2$ cannot be factored in $\Vec{A}$ on the left and $t'_1 \in dA$. Then:
\[
t \otimes m \otimes t' =
(t_1 t_2) \otimes m \otimes (t_1' t_2') =
t_1 \otimes (t_2.m.t_1') \otimes t_2'
\]
\end{proof}

We will now establish the unicity of reduced forms, starting with elementary tensors in reduced form before extending the result.

\begin{lem}
If $a\otimes m \otimes b$ and $a'\otimes m' \otimes b'$ are reduced forms in $R[\K_{\Vec{X}}] \otimes C_n(\K_{\Vec{A}}) \otimes R'[\K_{\Vec{X}}]$, then: 

\[
a.m.b = a'.m'.b'
\text{ if and only if }
\begin{cases}
a = a' \\
m = m' \\
b = b'
\end{cases}
\]

\end{lem}

\begin{proof}
For every $s \in \Delta_n$, there exists $\sigma(s)$ that represents $m(s)$, and two dipaths $\gamma, \delta$ such that $a.m.b (s)$ is represented by $(\gamma \cdot \sigma(s) ) \cdot \delta$,
with $[\gamma] = a$ and $[\delta] = b$.

In the same way, there also exists a path  $\sigma'(s)$ representing $m'(s)$ and two dipaths $\gamma', \delta'$ so that $a'.m'.b'(s)$ is represented by $(\gamma' \cdot \sigma'(s) ) \cdot \delta'$,  with $[\gamma'] = a$ and $[\delta'] = b$.

Since $a.m.b(s) = a'.m'.b'(s)$, there exists a continuous increasing bijective function
\[
f_s : [0,1] \to [0,1]
\]
\noindent such that 
\[
    (\gamma \cdot \sigma(s) ) \cdot \delta =  
    \left( (\gamma' \cdot \sigma'(s) ) \cdot \delta' \right) \circ f_s \]

 However, as in the proof of the existence of reduced forms,
 \begin{align*}
     \inf \left\{t 
 \; \vline \;
 \delta (t) \in A
 \right\}
 =\frac{1}{2} 
 &= 
 \inf \left\{t 
 \; \vline \;
 \left( (\gamma' \cdot \sigma'(s) ) \cdot \delta' \right) \circ f_s \in A
 \right\}\\
 &=  f_s^{-1} \left( \inf \left\{t 
 \; \vline \;
 \left( (\gamma' \cdot \sigma'(s) ) \cdot \delta' \right) \in A
 \right\}
 \right)
 \\
 &= f_s^{-1} \left(\frac{1}{2}\right)
 \end{align*}

 so we necessarily have $f_s(\frac{1}{2}) = \frac{1}{2}$, and thus $f_s$ induces a bijective reparametrization of $\delta'$ into $\delta$, which proves that $b=b'$.

The same reasoning gives $a = a'$, and that the restriction of $f_s$ to $[f_s^{-1} (\frac{1}{4}), f_s^{-1} (\frac{1}{2})]$ sends $\sigma'(s)$ to $\sigma(s)$, thus $m(s) = m'(s)$ for all $s \in \Delta_n$ and thus $m = m'$.
\end{proof}

\begin{rem}
In the case of precubical sets, only the case $n=0$ requires conditions on relative pairs, as the bimodules of $n$-chains are free for $n \geq 1$ \cite{eric}. This idea generalizes to our case.

Note that the separatingness of $\Vec{A}$ in $\Vec{X}$ is only needed in the case where the simplexes in the trace spaces may be constant. As such, in the case where $n \geq 1$,  normalizing our chain complexes (i.e. quotienting out the degenerate simplices)  would allow us to weaken the hypotheses considerably, as only the first homology bimodule would require special treatment. 

More precisely, we believe that, for $n \geq 1$, $R[\K_{\Vec{X}}] \otimes C_n(\K_{\Vec{A}}) \otimes R'[\K_{\Vec{X}}]$ is the direct sum of 

$1.$ The bimodule generated by the non-constant traces, which is free.

$2.$ The bimodule generated by the constant traces, which is likely projective under the hypotheses above, and which can be killed by normalizing singular chains before taking the homology.
\end{rem}


\begin{prop}
\label{prop:24}
If $(\Vec{A}, \Vec{X})$ is a relative pair then the transfer map is injective.
\end{prop}

\begin{proof}
Let $m \in \ker \iota$. It is possible to write $m$ as a finite sum of the form:
\[
\sum_{i \in I} (\lambda_i t_i) \otimes m_i \otimes (\mu_i t'_i)
\]
where $m_i$ is a trace simplex of $\Vec{A}$,  $t,t'$ are traces in $\Vec{X}$, $\lambda_i \in R$, $\mu_i \in R'$ , and
\[
\iota(m) = \sum_{i \in I} (\lambda_i .t_i). m_i . (\mu_i.t'_i) = 0
\]
This last equality happens in $C_n(\K_{\Vec{X}})$, which is a free $R$-$R'$-bimodule over simplexes of the trace space, so this equality is equivalent to the conjunction of its projections on all simplexes of the trace space of $\Vec{X}$. As such, take some $i_0 \in I$ and consider 
\[
J = \left\{i \in I \; | \; t_i.m_i.t'_i = t_{i_0}.m_{i_0}.t'_{i_0}\right\}
\]
Then 
\[
\sum_{j \in J} \lambda_j .(t_j. m_j . t'_j).\mu_j = 0
\]
But then, for every $j \in J$, $t_j \otimes m_j \otimes t'_j$ has a reduced form $a_j \otimes n_j \otimes b_j$ and
\[ a_j. n_j . b_j = t_j.m_j. t'_j 
= t_{i_0}.m_{i_0}.t'_{i_0} 
= a_{i_0}.n_{i_0}.b_{i_0} \]
\noindent so, according to the previous lemma, 
\begin{align*}
    a_j &= a_{i_0}\\
    n_j &= n_{i_0}\\
    b_j &= b_{i_0}
\end{align*}
\noindent and thus 
\begin{align*}
\sum_{j \in J} (\lambda_j .t_j) \otimes m_j \otimes  (\mu_j t'_j)
&= \sum_{j \in J} \lambda_j .(t_j \otimes m_j \otimes t'_j).\mu_j \\
&= \sum_{j \in J} \lambda_j .(a_j \otimes n_j \otimes b_j ). \mu_j\\
&= \sum_{j \in J} \lambda_j .(a_{i_0} \otimes n_{i_0} \otimes b_{i_0} ). \mu_j\\
\end{align*}
Then, since $C_n(\K_{\Vec{X}})$ is a free $R$-$R'$-bimodule with basis the trace simplices, the sub-$R$-$R'$-bimodule $R.(a_{i_0}.n_{i_0}.b_{i_0}).R'$ is a free $R$-$R'$-bimodule with basis $(a_{i_0}.n_{i_0}.b_{i_0})$. As such, there exists a unique $R$-$R'$-bimodule morphism 
\[f: 
R.(a_{i_0}.n_{i_0}.b_{i_0}).R' \to R.(a_{i_0} \otimes n_{i_0} \otimes b_{i_0} ).R', \quad
a_{i_0}.n_{i_0}.b_{i_0}\mapsto a_{i_0} \otimes n_{i_0} \otimes b_{i_0}
\]
\noindent and then
\[
\sum_{j \in J} \lambda_j .(a_{i_0} \otimes n_{i_0} \otimes b_{i_0} ). \mu_j 
= 
f\left( \sum_{j \in J} \lambda_j .(a_{i_0} . n_{i_0} . b_{i_0} ). \mu_j \right) = f(0) = 0
\]
Hence, for any $i_0 \in I$,
\[\sum_{j \in I, \; t_j.m_j.t_j = t_{i_0}.m_{i_0}.t'_{i_0}} (\lambda_j .t_j) \otimes m_j \otimes (t'_j . \mu_j ) = 0\]
and thus $m = 0$.
\end{proof}

In conclusion, when $(\Vec{A}, \Vec{X})$ is a relative pair then the transfer map is an isomorphism between $R[\K_{\Vec{X}}] \otimes C_n(\K_{\Vec{A}}) \otimes R'[\K_{\Vec{X}}]$ and $C_n^{\K_{\Vec{X}}} (R, \K_{\Vec{X}}, R')$.

As such, we obtain an exact sequence of chain complexes:

\[\begin{tikzcd}
	0 & R[\K_{\Vec{X}}] \otimes C_*(\K_{\Vec{A}}) \otimes R'[\K_{\Vec{X}}] & C_*^{\K_{\Vec{X}}}(R, \K_{\Vec{A}}, R') & 0
	\arrow[from=1-1, to=1-2]
	\arrow["\iota", from=1-2, to=1-3]
	\arrow[ from=1-3, to=1-4]
\end{tikzcd}\]
and the chain complexes are isomorphic. This allows us to simplify the exact sequence at the end of the last section:

\begin{prop}
There is a short exact sequence of chain complexes:
\[\begin{tikzcd}
	0 & {R[\cS] \otimes C_{*}(\cT) \otimes R'[\cS]} & {C_*(R, \cS, R')} & {C_*(R, \cS/ \cT, R')} & 0
	\arrow[from=1-1, to=1-2]
	\arrow[from=1-2, to=1-3]
	\arrow[from=1-3, to=1-4]
	\arrow[from=1-4, to=1-5]
\end{tikzcd}\]
As such, there is a long exact sequence of homology groups:

\[\begin{tikzcd}
	&& \vdots \\
	H_{n+1} \left(  {R[\cS] \otimes C_{*}(\cT) \otimes R'[\cS]}\right) &&
	\G_{n+1}(\K_{\Vec{X}}) &&
	\G_{n+1}(\K_{\Vec{X}}/ \K_{\Vec{A}}) \\
	H_n\left( {R[\cS] \otimes C_{*}(\cT) \otimes R'[\cS]}\right) &&
	\G_{n}( \K_{\Vec{X}}) && 
	\G_n(\K_{\Vec{X}}/ \K_{\Vec{A}})\\
	H_{n-1}\left(  {R[\cS] \otimes C_{*}(\cT) \otimes R'[\cS]}\right) &&
	\G_{n-1}(\K_{\Vec{X}}) && 
	\G_{n-1}(\K_{\Vec{X}}/ \K_{\Vec{A}}) \\
	&& \vdots & {}
	\arrow[from=1-3, to=2-1]
	\arrow["{j_*}", from=2-1, to=2-3]
	\arrow["{p_*}", from=2-3, to=2-5]
	\arrow[dashed, "{\delta_{n+1}}", from=2-5, to=3-1]
	\arrow["{j_*}"', from=3-1, to=3-3]
	\arrow["{p_*}"', from=3-3, to=3-5]
	\arrow[dashed, "{\delta_{n}}", from=3-5, to=4-1]
	\arrow["{j_*}"', from=4-1, to=4-3]
	\arrow["{p_*}"', from=4-3, to=4-5]
	\arrow[from=4-5, to=5-3]
\end{tikzcd}\]
\end{prop}



\section*{Credit author statement}

{\bf Eliot M\'edioni}: Conceptualization, Writing - original draft;
{\bf Eric Goubault}: Conceptualization, Writing - Review \& Editing, Supervision.

\section*{Acknowledgments}
The authors thank J\'er\'emy Dubut, Augustin Albert and Sanjeevi Krishnan, for numerous interactions during the preparation of this work.

\section*{Declaration of competing interests and funding}

The authors have no competing interests to report in the preparation of this work. 

This research did not receive any specific grant from funding agencies in the public, commercial, or not-for-profit sectors.

\bibliographystyle{cas-model2-names}
\bibliography{biblio}

\end{document}